\documentclass[10pt]{amsart}

\usepackage[T1]{fontenc}
\usepackage{amsmath,amssymb,mathtools}
\usepackage{mathrsfs}
\usepackage{microtype}
\usepackage{tikz}
\usepackage{tikz-cd}
\usepackage{enumitem}
\usepackage[margin=0.95in]{geometry}
\usepackage[colorlinks=true,citecolor=blue,linkcolor=blue,urlcolor=blue]{hyperref}

\newtheorem{theorem}{Theorem}[section]
\newtheorem{proposition}[theorem]{Proposition}
\newtheorem{lemma}[theorem]{Lemma}

\theoremstyle{definition}
\newtheorem{definition}[theorem]{Definition}

\theoremstyle{remark}
\newtheorem{remark}[theorem]{Remark}
\newtheorem*{theoremA}{Theorem A}
\newtheorem*{theoremB}{Theorem B}

\def\C{{\mathbb{C}}}
\def\Z{{\mathbb{Z}}}

\def\R{{\mathcal{R}}}
\def\O{{\mathcal{O}}}
\def\T{{\mathcal{T}}}

\def\K{{\mathcal{K}}}

\def\L{{\Lambda}}
\def\F{{\mathcal{F}}}

\def\P{{\mathcal{P}}}
\newcommand{\KP}{\mathcal{KP}}

\newcommand{\vloop}{v}

\def\G{{\tilde{G}}}
\def\Gr{\mathrm{Gr}}

\def\Coh{\mathrm{Coh}}

\newcommand{\Hom}{\mathrm{Hom}}

\newcommand{\conv}{\mathbin{*}}
\newcommand{\nconv}{\mathbin{\odot}}
\newcommand{\cl}{\mathrm{cl}}

\title[Categorified Cluster Structures]
{Categorified Cluster Structures\\
on Coulomb Branches of\\
Certain Star-Shaped Quivers}
\author{Tianle Liu}
\address{Department of Mathematics, University of Southern California}
\email{tianleli@usc.edu}

\date{}
\makeatletter
\@ifundefined{subjclassname@2020}{%
  \@namedef{subjclassname@2020}{%
    \textup{2020} Mathematics Subject Classification}%
}{}
\makeatother
\hypersetup{
  pdftitle={Categorified Cluster Structures on Coulomb Branches of Certain Star-Shaped Quivers},
  pdfauthor={Tianle Liu}
}

\begin{document}

\begin{abstract}
We study the quantized $K$-theoretic Coulomb branches of framed
star-shaped quivers with central gauge dimension $n\geq2$,
$r-1$ one-dimensional outer gauge vertices, and a one-dimensional
framing at the central vertex.
For $n=2$ and arbitrary $r\geq1$, and for every $n\geq2$ with
$1\leq r\leq5$, we prove that the loop-graded Koszul-perverse heart
is a monoidal categorification of the associated quantum cluster algebra
over $\Z[\vloop^{\pm1}]$.
This quantum cluster algebra coincides with its upper quantum cluster algebra.
For all $n\geq2$ and $r\geq1$, we construct an explicit initial seed
and identify the associated upper quantum cluster algebra with the
Grothendieck ring of this heart.
Every mutation is realized by a short exact sequence, and every
quantum cluster monomial is represented by a simple object.
\end{abstract}

\maketitle
\enlargethispage{4pt}

\section{Introduction}

For integers $n\geq2$ and $r\geq1$, let
\begin{equation}
\label{eq:intro-pair}
 G_{n,r}
 =\operatorname{im}\bigl(GL_n\times(\C^*)^r
   \longrightarrow GL(N_{n,r})\bigr),
 \qquad N_{n,r}=(\C^n)^{\oplus r},
\end{equation}
where $(g,z_1,\ldots,z_r)$ acts on the $a$th summand of $N_{n,r}$ by $gz_a$.
This is the gauge theory of a framed star-shaped quiver with central
dimension $n$, $r-1$ one-dimensional outer vertices, and a
one-dimensional framing at the central vertex
(Section~\ref{sec:group}).
The loop-rotation equivariant $K$-theory of the classical space of triples
$\R_{G_{n,r},N_{n,r}}^{\cl}$ \cite{Nak16,BFN18,BFN19} quantizes the
$K$-theoretic Coulomb branch of this pair.

Cautis and Williams work with the derived space of triples
$\R_{G_{n,r},N_{n,r}}$ and equip its equivariant coherent category
with a Koszul-perverse $t$-structure. Its heart is a rigid
monoidal category of finite length with renormalized $r$-matrices
\cite[Thms.~4.1, 5.1, 6.1, and 7.1]{CW23}, and its simple objects
give a positive integral basis of the quantized Coulomb branch
\cite[Cor.~1.3]{CW23}.
We prove that the graded heart is a monoidal categorification of a quantum
cluster algebra when $n=2$ and $r\geq1$, and when $n\geq2$ and
$1\leq r\leq5$. The identification of its Grothendieck ring with the
upper quantum cluster algebra and the categorical realization of all
mutations hold for every $n\geq2$ and $r\geq1$.

Let $\mathcal C_{n,r}$ denote this heart, with the loop-rotation
grading and the shift conventions specified in
Sections~\ref{sec:loop-rotation-cover} and~\ref{sec:gln-signs-upper}.
Write $\vloop$ for the parameter recording this grading.
To describe the initial seed, use the coordinates $A=gz_1$ and
$h_a=z_1/z_a$ for $a\geq2$, which identify
$G_{n,r}$ with $GL_n\times(\C^*)^{r-1}$.
Let $\omega_k$ be the coweight that acts by $t$ on the first $k$
coordinates of the standard $GL_n$ representation and trivially on
the remaining coordinates and outer torus factors. Let $\gamma_a$
act by $t$ on the $a$th summand of $N_{n,r}$ and trivially on the others.

The simple object $\mathcal P_{\gamma_a,0}$ corresponds to the trivial
equivariant line bundle on the point orbit of $\gamma_a$ in the affine
Grassmannian.
The simple object $\mathcal P_{\omega_k,\ell}$ corresponds to
$(\det Q_k)^\ell$, where $Q_k$ is the rank-$k$ tautological quotient
on the orbit of $\omega_k$.
In both cases, the line bundle is pulled back to the classical support
in the space of triples and pushed forward with the cohomological
and grading shifts specified in Section~\ref{sec:gln-signs-upper}.
The initial mutable objects are
\begin{equation}
\label{eq:intro-general-A0}
 \mathcal P_{\gamma_a,0}\quad(1\leq a\leq r),\qquad
 \mathcal P_{\omega_k,0},\ \mathcal P_{\omega_k,1}
                    \quad(1\leq k<n).
\end{equation}
The one-dimensional representations $\det A,h_2,\ldots,h_r$ similarly
determine $r$ frozen, tensor-invertible objects, with shifts given in
the same section. Write $\mathcal A_{n,r}$ and $\mathcal U_{n,r}$ for
the quantum cluster algebra and upper quantum cluster algebra of this
seed over $\Z[\vloop^{\pm1}]$, with these frozen variables inverted.
A simple object is called real if its convolution square is simple.

\pagebreak[3]
\begin{theoremA}
Suppose either $n=2$ and $r\geq1$, or $n\geq2$ and $1\leq r\leq5$.
There is an explicit compatible pair
$(\Lambda_{n,r},\widetilde B_{n,r})$ with $2n+r-2$ mutable and
$r$ invertible frozen directions. Sending the initial mutable and
frozen variables to the classes of the objects in
\eqref{eq:intro-general-A0} and the objects associated with $\det A,h_2,\ldots,h_r$ defines
an isomorphism of $\Z[\vloop^{\pm1}]$-algebras
\[
 \iota\colon\mathcal A_{n,r}=\mathcal U_{n,r}
 \xrightarrow{\ \sim\ } K_0(\mathcal C_{n,r}).
\]
Under this isomorphism:
\begin{enumerate}[label=\textup{(\roman*)}]
 \item every quantum cluster variable is represented by a real simple
 Koszul-perverse object;
 \item every quantum cluster monomial is represented by a simple
 convolution product with the grading shift prescribed by the
 quantum-torus multiplication; and
 \item every finite mutation sequence is realized by quantum monoidal
 seeds and short exact sequences in the Koszul-perverse heart.
\end{enumerate}
Thus $\mathcal C_{n,r}$ is a monoidal categorification of
$\mathcal A_{n,r}$ in these ranges.
See Theorem~\ref{thm:family-categorification}.
\end{theoremA}

\begin{theoremB}
For every $n\geq2$ and $r\geq1$, the initial simple objects determine
an integral compatible quantum seed and an isomorphism
\[
 \mathcal U_{n,r}\xrightarrow{\sim}K_0(\mathcal C_{n,r}).
\]
Every finite mutation sequence is realized in the heart, and every
quantum cluster monomial is represented by a simple convolution
product with the grading shift prescribed by the quantum-torus
multiplication.
See Theorems~\ref{thm:gln-upper-Coulomb} and~\ref{thm:cat-all-mutations}.
\end{theoremB}

The generation assertion $\mathcal A_{n,r}=\mathcal U_{n,r}$ is thus
the remaining problem for $n\geq3$ and $r\geq6$.

For $n=2$, the initial seed is
\begin{equation}
\label{eq:intro-A0}
 \mathsf A_0=
 (\P_{\omega_1,0},\P_{\gamma_1,0},\ldots,\P_{\gamma_r,0},\P_{\omega_1,1}).
\end{equation}
Its principal quiver has one arrow from each $\P_{\gamma_a,0}$ to
$\P_{\omega_1,1}$, a double arrow from $\P_{\omega_1,1}$ to $\P_{\omega_1,0}$,
and one arrow from $\P_{\omega_1,0}$ to each $\P_{\gamma_a,0}$.
When $n=2$ and $r\leq3$, the pattern contains an acyclic seed, so
$\mathcal A_{2,r}=\mathcal U_{2,r}$ also follows from
\cite[Thm.~7.5]{BZ05}. For $r=4$ its principal part has
the four-punctured-sphere mutation type $D_4^{(1,1)}$ and has no
acyclic seed; for $r\geq5$ it is mutation infinite
(Remark \ref{rem:acyclic-small-rank}).

\subsection{Outline of the proof}
\label{sec:intro-outline}

The construction of the seed and its categorical mutations is uniform
in $n$ and $r$. The two ranges in Theorem A arise in the generation
argument.
Section~\ref{sec:group} describes the classical loci over minuscule
and zero-dimensional affine-Grassmannian orbits using lattices. Section~\ref{sec:gln-seed} constructs
an integral compatible pair with
$\widetilde B_{n,r}^{\mathsf T}\Lambda_{n,r}=(2I\;0)$.
The Whittaker transform of \cite[Thm.~5.13]{SS19} identifies the
quantized universally Laurent $X$-algebra with the Coulomb branch.
We adjust the signs in its quantum torus to match the classes of our
initial simple objects. This adjustment is compatible with mutation.
Comparison of the Laurent charts over the integral coefficient ring
then identifies the upper quantum cluster algebra with the
Grothendieck ring (Theorem~\ref{thm:gln-upper-Coulomb}).

Section~\ref{sec:category} constructs the initial exchange sequences
from the coherent Q-system and the sequences associated with Cartier
divisors in \cite[Props.~8.5 and 8.8]{CW23}. The additional summands
of $N_{n,r}$ contribute vector bundles on the classical supports of the
convolution correspondences. Proper flat base change and the projection
formula give the corresponding exchange sequences for this family.
The initial seed therefore admits mutation in every mutable direction.
The mutation criterion of
\cite[Thm.~4.34(ii)]{CW19}, in the homogeneous form verified in
Lemma~\ref{lem:frozen-localized-factorization}, then extends these
exchange sequences along every finite mutation path.

Section~\ref{sec:integral-generation} proves generation by cluster
variables in the two ranges of Theorem A. The required Coulomb-branch
generators are monopole operators twisted by elements of the
representation rings of their stabilizer Levi subgroups; these twists
are called dressings. Tensoring by powers of the relative determinant
line bundle, with the prescribed shifts, and applying the Baxter sequences
of seed mutations of \cite[Sec.~5.3]{SS19} first give the monopole operators
indexed by $\pm\gamma_a$ and all dressings by integer powers of $\det Q_k$
for $\omega_k-\sum_{a\in S}\gamma_a$, where $S\subseteq\{1,\ldots,r\}$.
We then use convolution identities to obtain dressings by arbitrary symmetric
Laurent polynomials in the $k$ weights of $Q_k$.
For $k=2$, multiplication by the monopole operators indexed by $-\gamma_a$
gives the result for arbitrary $r$.

For $k\geq3$, we argue by induction on $k$. Linear combinations of
products of the operators obtained for $k-1$ and $k=1$ give all
dressings in an ideal. The quotient of the dressing ring by this
ideal parametrizes monic degree-$k$ factors of a polynomial with
one linear factor for each $a\in S$. A further convolution identity
replaces this polynomial by the product with one linear factor omitted.
When $r\leq5$, the polynomial to be factored has degree at
most four, so the resulting quotient ring is spanned over the parameter
ring by integer powers of the product of the $k$ weights of $Q_k$.
The corresponding operators already belong to the cluster algebra,
completing the induction.
At $k=n$, the same identities produce all representation characters
of $G_{n,r}$. Together with the preceding dressings, these give all
dressings by elements of Levi representation rings required by the
generation result of \cite[Prop.~2.3 and the discussion following it]{SS19}.
The decomposition of coweights in
Lemma~\ref{lem:gen-chamber} then completes the proof.

\subsection{Related work}

We use the quantum cluster algebras of \cite{BZ05}, upper cluster
algebras of \cite{BFZ05}, and cluster ensembles of \cite{FG09}.
Schrader and Shapiro \cite[Thm.~5.13]{SS19} identify the quantized
universally Laurent $X$-algebra with the Coulomb branch.
We compare their algebra with the integral upper quantum cluster
algebra whose initial variables are the classes of the simple objects
above, and prove generation by cluster variables in the ranges of
Theorem A.

Allegretti, Kim, and Shan \cite[Thm.~1.1 and Sec.~1.2]{AKS25} identify
relative genus-zero skein algebras with quantized Coulomb branches
over complex Laurent parameter rings and deduce categorifications
by Koszul-perverse hearts.  Their gauge data use products of $SL_2$
with a separate flavor torus. They also propose using their methods
to construct monoidal categorifications of surface cluster algebras
\cite[Sec.~1.3.3]{AKS25}.
For the family \eqref{eq:intro-pair}, Theorems A and B identify cluster
monomials with simple objects and realize mutations by short exact sequences.

For triangulable surfaces with at least two punctures,
Li \cite[Cor.~1.2 and Secs.~1.1--1.2]{Li25}
identifies the classical tagged skein algebra with the upper cluster
algebra and constructs a skein quantization.

Monoidal categorification originates with Hernandez and Leclerc
\cite{HL10}.  For symmetric Khovanov--Lauda--Rouquier (KLR) algebras,
Kang, Kashiwara, Kim, and Oh
prove that admissible pairs remain admissible after mutation,
assuming that the associated quantum cluster algebra and the
Grothendieck ring become isomorphic after extending scalars to
$\mathbb Q(q^{1/2})$, in their notation
\cite[Eq.~(7.4), Thm.~7.1.3, and Cor.~7.1.4]{KKKO18}.
The simplicity of heads and socles of tensor products with a real
factor, used in this argument, is proved in \cite{KKKO15}.
Renormalized $R$-matrices go back to
spectral-parameter arguments in level-zero quantum affine theory
\cite{Kas02}.

The argument realizing successive mutations uses the factorization $r$-matrices of
\cite[Thm.~5.10]{CW19} and adapts the proof of
\cite[Thm.~4.34(ii)]{CW19} inside the graded heart.

This article extends and strengthens the results of the author's
thesis~\cite{LiuThesis}.

\subsection{Conventions}

All schemes, groups, and representations are over $\C$; we write
$\O=\C[[t]]$ and $\K=\C((t))$.
The double-cover loop-rotation convention of Section
\ref{sec:loop-rotation-cover} supplies the grading, and $\vloop$
denotes its quantum parameter. The objects associated with $\det A,h_2,\ldots,h_r$ are inverted in the monoidal seed. The symbols
$\omega_k,\gamma_1,\ldots,\gamma_r$ denote chosen coweight lifts unless
brackets indicate their images in the fundamental group.

\section{Preliminaries}
\label{sec:preliminaries}

We recall the derived space of triples, the Koszul-perverse convolution category,
and the structural results of Cautis and Williams \cite{CW23}, in the grading
convention used below. Their construction refines the classical Coulomb branch
of Braverman, Finkelberg, and Nakajima \cite{BFN18}. An
unsubscripted pair $(\G,N)$ denotes a connected reductive group with a
finite-dimensional representation.  Our applications concern the family
$(G_{n,r},N_{n,r})$ of \eqref{eq:intro-pair}.

\subsection{The loop-rotation cover}
\label{sec:loop-rotation-cover}

Let $\C^\times_{\mathrm{geom}}$ be the ordinary loop-rotation torus, acting on
the loop coordinate $t$ with weight one.  Throughout this paper
$\C^\times_{\mathrm{rot}}$ instead denotes its double cover
$\pi_{\mathrm{rot}}\colon\C^\times_{\mathrm{rot}}\to\C^\times_{\mathrm{geom}}$,
$c\mapsto c^2$, and every geometric rotation action is pulled back along
$\pi_{\mathrm{rot}}$; thus $t$ has rotation weight two.  We write
$\chi_{\mathrm{rot}}$ for the defining character of the covering torus, put
$\vloop=\chi_{\mathrm{rot}}^{-1}$, and define $M\{r\}$ by tensoring with
$\chi_{\mathrm{rot}}^{-r}$, so $[M\{r\}]=\vloop^r[M]$.  This is the
double-cover convention used in
\cite[p.~4, paragraph preceding Def.~2.1]{SS19}; we identify the corresponding
double-cover character $q_{\mathrm{SS}}$ with $\chi_{\mathrm{rot}}$.  The
quadratic change of signs in \eqref{eq:gln-rho} of Section
\ref{sec:gln-signs-upper} removes the $-1$ cocycle in that torus convention and
completes the identification; after it $\vloop=q_{\mathrm{SS}}^{-1}$, and
$\vloop=1$ is the ordinary commutative specialization.  Thus
$K_{\C^\times_{\mathrm{rot}}}(\mathrm{pt})=\Z[\chi_{\mathrm{rot}}^{\pm1}]
\cong\Z[q_{\mathrm{SS}}^{\pm1}]$, whereas the defining character $Q$ of
ordinary rotation satisfies $Q\mapsto\chi_{\mathrm{rot}}^2=\vloop^{-2}$.

Pullback from ordinary rotation to the double cover is a finite central change
of equivariance whose kernel $\mu_2$ acts trivially on all spaces.  The
central $\mu_2$-action splits every equivariant coherent category into even
and odd character summands; the even summand descends to ordinary rotation
equivariance, and tensoring by $\chi_{\mathrm{rot}}$ identifies the two.
Hence the coherent categories, their t-structures, convolution functors,
duality, and factorization maps extend character by character, and applying
this decomposition at every finite stage and passing to the defining filtered
colimit identifies the effect on equivariant $K$-groups with the scalar
extension
\begin{equation}
\label{eq:rotation-cover-base-change}
 K^{\C^\times_{\mathrm{geom}}}(X)
 \mathop{\otimes}_{\Z[Q^{\pm1}]}
 \Z[\chi_{\mathrm{rot}}^{\pm1}]
 \xrightarrow{\sim}
 K^{\C^\times_{\mathrm{rot}}}(X).
\end{equation}
Consequently the structural theorems used below remain valid
after this base change.  Definition \ref{def:half-shift-extension} formally
adjoins half-integer grading shifts.

We fix the cohomological normalization for the whole family.
Cautis and Williams use formal degrees $\Z+\tfrac12$ on components with
odd-dimensional Schubert orbits and the fourfold rotation cover for their
bar-invariant simple objects \cite[Sec.~2.1, Rem.~2.4]{CW19}.
For $G_{n,r}\simeq GL_n\times(\C^*)^{r-1}$, define
\[
 c_{n,r}:\pi_1(G_{n,r})\longrightarrow\Z,
 \qquad c_{n,r}([g])=(n-1)\operatorname{val}_t\det A.
\]
The value of this homomorphism has the same parity as the dimension
of the corresponding Schubert orbit: for a dominant
coweight with central coordinates $\lambda_1\geq\cdots\geq\lambda_n$,
\[
 \dim\Gr_\lambda=\sum_{i=1}^n(n+1-2i)\lambda_i
 \equiv(n-1)\sum_i\lambda_i\pmod2.
\]
Write $\langle s\rangle=[s]\{-s\}$ in the formal convention and,
on component $\beta$, define the functor
\begin{equation}
\label{eq:integral-component-transport}
 \mathfrak T_\beta(\mathcal F)
 :=\mathcal F\langle-c_{n,r}(\beta)/2\rangle
 =\mathcal F[-c_{n,r}(\beta)/2]\{c_{n,r}(\beta)/2\}.
\end{equation}
More precisely, on component $\beta$ take the summand of fourfold-cover
equivariance on which the kernel of the map to the double cover acts by
$(-1)^{c_{n,r}(\beta)}$. This summand contains the bar-invariant simple
objects and their integral loop shifts. Formula~\eqref{eq:integral-component-transport}
makes its cohomological degrees integral and its kernel action trivial,
so it descends to coherent complexes with double-cover equivariance.
Both the component index and the kernel character add under convolution.
Consequently the functors $\mathfrak T_\beta$ admit coherent isomorphisms
\[
 \mathfrak T_{\beta+\beta'}(\mathcal F\conv\mathcal G)
 \simeq\mathfrak T_\beta(\mathcal F)\conv
             \mathfrak T_{\beta'}(\mathcal G),
\]
and together define a monoidal equivalence to the integral convention used here.
The normalizations of our initial simple objects are specified in
Section~\ref{sec:gln-signs-upper}.
In particular, $c_{n,r}([\omega_k])=k(n-1)$,
$c_{n,r}([\gamma_1])=n(n-1)$, and
$c_{n,r}([\gamma_a])=0$ for $a\geq2$.
Thus even a component with a zero-dimensional Schubert orbit can acquire
a nontrivial shift. Applying the inverse functors $\mathfrak T_\beta^{-1}$
after adjoining a square root of the loop character recovers the formal
convention above.

We also compare rotation actions by changing the splitting of rotation
from the gauge centre. Schrader and Shapiro use
\[
 (g,c)\cdot s(t)=g\,c\,s(c^2t)
\]
\cite[Rem.~2.4]{SS19}, whereas our rotation has no scalar factor on $N$.
The cocharacter $\eta(c)=(cI_n,1,\ldots,1)$ acts with weight one on every
summand of $N_{n,r}$. Pullback along
\begin{equation}
\label{eq:SS-rotation-centre-comparison}
 \psi(g,c)=(g\eta(c^{-1}),c)
\end{equation}
identifies their action with ours on the entire convolution diagram,
and therefore induces a monoidal equivalence and a convolution-ring
isomorphism. On characters it sends
\[
 q_{\mathrm{SS}}\longmapsto\vloop^{-1},\qquad
 w_i^{\mathrm{SS}}\longmapsto\vloop w_i^{\mathrm{geom}}
       \quad(1\leq i\leq n),\qquad
 h_a^{\mathrm{SS}}\longmapsto h_a\quad(a\geq2).
\]
Here $w_i^{\mathrm{geom}}$ are the $GL_n$ weights and $h_a$ are the
characters of the remaining gauge factors. A character of weight $m$
for $\eta$ therefore acquires the loop factor $\vloop^m$.
Every use of the Whittaker isomorphism below includes pullback along $\psi$.
The symbols $w_i$ in localization formulas denote the pulled-back
characters; in particular $\prod_iw_i=\vloop^n\det A$.

\subsection{The convolution category}
\label{sec:tame-foundations}

Put
\[
 \widehat{\G}_{\O}:=(\G_{\O}\rtimes\C^\times_{\mathrm{rot}})\times\C^\times_{\mathrm{sc}},
 \qquad
 \widehat{\G}_{\K}:=(\G_{\K}\rtimes\C^\times_{\mathrm{rot}})\times\C^\times_{\mathrm{sc}},
\]
where $\C^\times_{\mathrm{sc}}$ acts on $N$ by scalar multiplication, and set
\[
 Y=N_\O/\widehat{\G}_\O,\qquad Z=N_\K/\widehat{\G}_\K,\qquad
 X=Y\mathop{\times}^{\mathbf R}_Z Y.
\]
By \cite[Props.~4.7, 4.8, and 4.11]{CW23},
$X\simeq\R_{\G,N}/\widehat{\G}_\O$ is an admissible, ind-tamely
presented ind-geometric stack, and the two projections $X\to Y$
are ind-proper and almost ind-finitely presented.
We use the coherent pullback and pushforward functors of
\cite{CWInd23,CWTame23}. Their hypotheses for this correspondence,
including weak smoothness of $Y$, are verified in
\cite[Prop.~5.7]{CW23}; the associative, unital convolution structure
is \cite[Prop.~5.5]{CW23}. The change of rotation cover in
Section~\ref{sec:loop-rotation-cover} preserves these constructions.

\subsection{The derived space of triples}

Let $\mathcal W:=\Gr_\G\times N_\K$, an infinite-rank vector bundle over the
affine Grassmannian $\Gr_\G$.  It contains the trivial bundle
$\mathcal V_1:=\Gr_\G\times N_\O$, embedded by
$j_1\colon\mathcal V_1\hookrightarrow\mathcal W$ along
$N_\O\hookrightarrow N_\K$, and the twisted bundle
$\T_{\G,N}:=\G_\K\times_{\G_\O}N_\O$, embedded by
$j_2\colon[g,s]\mapsto([g],gs)$.  The \emph{derived space of triples}
$\R_{\G,N}$ is the derived intersection of these two bundles, the fiber product
in the category of derived ind-schemes \cite[Def.~4.2]{CW23}:
\[
\begin{tikzcd}
\R_{\G, N} \arrow[r, "i_1"] \arrow[d, "i_2"'] & \Gr_\G \times N_\O \arrow[d, "j_1"] \\
\T_{\G, N} \arrow[r, "j_2"'] & \Gr_\G \times N_\K
\end{tikzcd}
\]
Write $\R_{\G,N}^{\mathrm{cl}}:=t_0(\R_{\G,N})_{\mathrm{red}}$ for the reduction
of the classical truncation.  It recovers the set-theoretic definition
\[ \R_{\G, N}^\mathrm{cl} = \{ ([g], s) \in \Gr_\G \times N_\O \mid g^{-1}s \in N_\O \}. \]

The derived structure can be seen on every finite Schubert and lattice
truncation.  There the two bundles are finite-rank subbundles of a finite-rank
quotient of $\mathcal W$, and the derived intersection is the derived zero
locus of $(x,y)\mapsto j_1(x)-j_2(y)$ on
$(\Gr_\G\times N_\O)\oplus\T_{\G,N}$, whose structure sheaf is the
corresponding Koszul complex. At a point $[g]\in\Gr_\G$, the classical
truncation of the fiber is $N_\O\cap gN_\O$, while the negative cohomology
of the Koszul complex records the excess quotient
$N_\K/(N_\O+gN_\O)$. Passing compatibly through the truncations of
\cite[Sec.~4.1 and Prop.~4.7]{CW23} gives $\O_{\R_{\G,N}}$.

\subsection{The Koszul-perverse heart and its grading}
\label{sec:kp-heart}

There are two independent one-dimensional equivariances: the covering torus
$\C^\times_{\mathrm{rot}}$ acts on $t$ with weight two as fixed in Section
\ref{sec:loop-rotation-cover}, while $\C^\times_{\mathrm{sc}}$ acts by scalar
multiplication on $N$.  Recall $\widehat{\G}_{\O}$ from Section
\ref{sec:tame-foundations} and write
$\Coh^{\widehat{\G}_{\O}}(\R_{\G,N}):=\Coh(\R_{\G,N}/\widehat{\G}_{\O})$ for
the stable infinity-category of coherent sheaves on the admissible, ind-tamely presented
quotient stack; for the adjoints needed for rigidity, \cite{CW23} works with
the corresponding equivariant $\operatorname{IndCoh}$ categories.  Let
$\omega_1\colon\Gr_{\G}\to\Gr_{\G}\times N_{\O}$, $[g]\mapsto([g],0)$, and
$\omega_2\colon\Gr_{\G}\to\mathcal T_{\G,N}$, $[g]\mapsto[g,0]$, be the zero
sections.  The construction uses the theory of perverse coherent sheaves
of \cite{AB10}, applied to the affine
Grassmannian in \cite{BFM05,CW19}; the regrading has the same formal origin
as the Koszul regradings of \cite{BGS96}.

For a general pair $(\G,N)$, the Koszul-perverse t-structure is constructed using the
formal cohomological convention recalled in Section
\ref{sec:loop-rotation-cover}.  For the remainder of this section fix
$n\geq2$, $r\geq1$, and put $(\G,N)=(G_{n,r},N_{n,r})$.
Apply the functors $\mathfrak T_\beta$ of \eqref{eq:integral-component-transport}.
For a dominant coweight $\lambda^\vee$, let
$\Gr_{\lambda^\vee}^\circ$ be its $\G_\O$-orbit and
$\R_{\lambda^\vee}^\circ$ the restriction of $\R_{\G,N}$ over this orbit.
Write $\R_{\lambda^\vee}^{\circ,\mathrm{cl}}
:=t_0(\R_{\lambda^\vee}^\circ)_{\mathrm{red}}$.

\begin{theorem}[Cautis--Williams, {\cite[Thm.~4.1]{CW23}}]
\label{thm:CW-heart}
There is a unique t-structure on $\Coh^{\widehat{\G}_{\O}}(\R_{\G,N})$ for
which
\[
 \omega_1^*i_{1*}\colon
 \Coh^{\widehat{\G}_{\O}}(\R_{\G,N})\longrightarrow\Coh^{\widehat{\G}_{\O}}(\Gr_{\G})
\]
is t-exact relative to the regraded perverse coherent t-structure on
$\Gr_{\G}$; equivalently, it is the unique t-structure for which
$\omega_2^*i_{2*}$ is t-exact.  This t-structure is bounded and has a
finite-length heart $\KP_{\G,N}$.

Up to loop and Koszul shifts, its simple objects are $\P_{\lambda^\vee,\mu}$
indexed by dominant pairs: $\lambda^\vee$ is a dominant coweight of $\G$, and
$\mu$ is a dominant weight for the Levi factor of the parabolic attached to
$\lambda^\vee$.  Equivalently, the labels form
$(P^\vee_{\G}\times P_{\G})/W_{\G}$.  The object $\P_{\lambda^\vee,\mu}$ is the
unique simple extension whose restriction to $\R_{\lambda^\vee}^\circ$ is
the pushforward from $\R_{\lambda^\vee}^{\circ,\mathrm{cl}}$ of the shifted
equivariant vector bundle corresponding to $\mu$.
\end{theorem}

Here we use \cite[Thm.~4.1]{CW23} after applying the functors
$\mathfrak T_\beta$ of \eqref{eq:integral-component-transport}. The functors $\omega_i^*i_{i*}$
preserve the component index and commute with componentwise shifts,
so the characterizing t-exactness conditions and the classification
of simple objects are unchanged.

The central cocharacter $\eta\colon\C^\times\to G_{n,r}$ given by
$\eta(c)=(cI_n,1,\ldots,1)$ in the coordinates $(A,h_2,\ldots,h_r)$ acts on
$N_{n,r}$ with weight one.  Identifying the scaling action with $\eta$ gives
the graded heart $\KP^\eta_{G_{n,r},N_{n,r}}$, abbreviated
$\mathcal C_{n,r}$. Sections 3.6 and 4.3 of \cite{CW23} decompose the heart
along the character summands of a trivially acting torus: the torus
$H_\eta=\{(\eta(c^{-1}),c):c\in\C^\times\}
\subset\G_{\O}\times\C^\times_{\mathrm{sc}}$
acts trivially on $\R_{\G,N}$, and
\[
 \KP_{\G,N}\simeq
 \bigoplus_{k\in\Z}\KP^\eta_{\G,N},
\]
where the summands are indexed by the $H_\eta$-degree and identified with
the degree-zero summand by powers of the Koszul shift
$\mathcal F\mapsto\mathcal F[1]\otimes\chi_{\mathrm{sc}}^{-1}$, which
lowers that degree by one.  Thus $\KP^\eta_{\G,N}$ has one
residual loop-grading parameter, written $\vloop$.  The homomorphism
$\G_{\O}\times\C^\times_{\mathrm{sc}}\to\G_{\O}$, $(g,c)\mapsto g\eta(c)$, has
kernel $H_\eta$ and is split by $g\mapsto(g,1)$.  Hence the degree-zero
$H_\eta$-character summand is canonically the category with
$\G_{\O}$-equivariance and no separate scaling equivariance:
\begin{equation}
\label{eq:eta-degree-zero-equivalence}
 \Coh^{\widehat{\G}_{\O}}(\R_{\G,N})_{H_\eta\text{-degree }0}
 \simeq
 \Coh^{\G_{\O}\rtimes\C^\times_{\mathrm{rot}}}(\R_{\G,N}).
\end{equation}
Thus $\KP^\eta_{\G,N}$ is the heart in the right-hand category of
(\ref{eq:eta-degree-zero-equivalence}); this is the present specialization of
\cite[Lem.~3.32]{CW23}.  Convolution adds this degree, the monoidal unit has
degree zero, duality negates it, and the factorization construction of the
renormalized $r$-matrices is homogeneous.  Consequently the degree-zero
summand is closed under convolution, both duals, and all $r$-matrices.

\subsection{Convolution as composition of kernels}
\label{sec:convolution-correspondence}

For the self-fiber product $X=Y\times_ZY$ of Section
\ref{sec:tame-foundations}, composition of kernels uses
\[
\begin{tikzcd}[column sep=large]
 &X\times_YX
   \arrow[dl,"\Delta"']\arrow[dr,"m"]&\\
 X\times X&&X ,
\end{tikzcd}
\]
where $\Delta$ remembers the two composable arrows and $m$ composes them.  For
coherent objects $\mathcal F,\mathcal G$, convolution is
\[
 \mathcal F\conv\mathcal G:=m_*\Delta^*(\mathcal F\boxtimes\mathcal G).
\]
\cite[Props.~5.5 and 5.7]{CW23} supply coherent pullback, pushforward,
and associativity. On classical points a representative of the correspondence
is $(g_1,g_2,s)$ with $s\in N_\O\cap g_1N_\O\cap g_1g_2N_\O$, and $m$ forgets
the intermediate modification $g_1$; in the lattice models this is the
passage from a chain $L_0,L_1,L_2$ to its endpoints $L_0,L_2$.

\begin{theorem}[Cautis--Williams, {\cite[Thms.~5.1, 6.1, 7.1]{CW23}}]
\label{thm:CW-monoidal-structure}
Left and right convolution by a Koszul-perverse object are t-exact, and the
monoidal unit is Koszul-perverse.  Hence $\KP^\eta_{\G,N}$ is a finite-length
monoidal abelian category.  It is rigid and carries the canonical homogeneous
system of renormalized $r$-matrices.
\end{theorem}

The calculation of $H_\eta$-degrees in Section~\ref{sec:kp-heart} shows that these
structures restrict to $\KP^\eta_{\G,N}$.

The Hom spaces in this heart are finite-dimensional. By \cite[Prop.~3.27 and Thm.~4.1]{CW23},
restriction to the open stratum followed by the functor $\omega^*p_*$
of the cited proposition embeds
$\operatorname{End}(S)$, for a simple object $S$, into the endomorphism
ring of its irreducible Levi representation. Schur's lemma therefore gives
$\operatorname{End}(S)=\C$. Finite length then gives
$\dim_{\C}\operatorname{Hom}(A,B)<\infty$ for all objects $A,B$ of the heart,
by induction on their lengths. The unit is the simple object
$\mathbf1=\P_{0,0}$ by \cite[Thms.~4.1 and 5.1]{CW23}.

\subsection{Grothendieck rings}
\label{sec:heart-K-theory}

The bounded t-structure and \eqref{eq:eta-degree-zero-equivalence} give
\begin{equation}
\label{eq:heart-K0}
 K_0\bigl(\KP^\eta_{\G,N}\bigr)
 \xrightarrow{\sim}
 K_0\!\left(\Coh^{\,\G_{\O}\rtimes\C^\times_{\mathrm{rot}}}(\R_{\G,N})\right).
\end{equation}
This is the isomorphism on $K_0$ supplied by the theorem of the heart
\cite[Thm.~6.1]{Barwick15}; the inverse is
$[\mathcal F]\mapsto\sum_i(-1)^i[H^i(\mathcal F)]$.
Convolution is t-exact by \cite[Thm.~5.1]{CW23}, so this is a ring
isomorphism.

We use the identification with classical equivariant $K$-theory in
\cite[Introduction and Cor.~1.3]{CW23}:
\begin{equation}
\label{eq:derived-classical-K-comparison}
 K_0\!\left(
   \Coh^{\G_{\O}\rtimes\C^\times_{\mathrm{rot}}}(\R_{\G,N})\right)
 \simeq
 K^{\G_{\O}\rtimes\C^\times_{\mathrm{rot}}}(\R_{\G,N}^{\mathrm{cl}}).
\end{equation}
The inverse is induced by the closed immersion
$j\colon\R_{\G,N}^{\mathrm{cl}}\hookrightarrow\R_{\G,N}$.
On each finite lattice stage, Postnikov and nilpotent filtrations give
this identification by d\'evissage. These maps commute with the flat
lattice pullbacks and closed pushforwards of
\cite[Eq.~(4.6) and Lem.~4.14]{CW23}, and hence pass to the space of triples.

The multiplication in \eqref{eq:derived-classical-K-comparison} agrees
with the classical Coulomb-branch convolution. In the correspondence
\cite[Prop.~5.9 and Eq.~(5.10)]{CW23}, the class of the derived pullback
on a finite lattice quotient is the alternating sum of the relative
Tor sheaves. This is the refined pullback with support used in the
classical construction \cite[Rem.~3.9(3)]{BFN18}; the proper pushforward
is the same. Both operations commute with the transition maps, so
\eqref{eq:derived-classical-K-comparison} is an isomorphism of convolution
rings. Together with \eqref{eq:heart-K0}, it identifies the quantized
$K$-theoretic Coulomb-branch algebra with $K_0(\KP^\eta_{\G,N})$.

\section{Quantum Seeds and Monoidal Mutation}
\label{sec:quantum-mutation}

We use the quantum-seed definitions of
\cite[Defs.~3.1, 4.1, 4.3, 4.5 and 4.8]{BZ05} and the monoidal
conventions of \cite[Secs.~4.3--4.4]{CW19}. A based quantum-torus
element is $X^c$, a toric frame is $M$, and a quantum seed is
$(M,\widetilde B)$. The parameter of \cite{BZ05} is
$q_{\mathrm{BZ}}=\vloop$.
The conversion from the commutation matrix
$L^{\mathrm{CW}}$ and parameter $q_{\mathrm{CW}}$
of \cite[Sec.~4.4]{CW19} is
\begin{equation}
\label{eq:BZ-CW-source-notation}
 L^{\mathrm{CW}}=-\Lambda,
 \qquad q_{\mathrm{CW}}=\vloop^{-1},
 \qquad
 L^{\mathrm{CW}}\widetilde B=\binom{D}{0}.
\end{equation}
The $X$-seed notation used by Schrader and Shapiro is introduced in Section
\ref{sec:SS-input}.

\subsection{Compatible quantum seeds}

Let $I=I_{\mathrm{ex}}\sqcup I_{\mathrm{fr}}$ be an ordered finite set
with the exchangeable indices first.  Write $e_i$ for the standard basis
of $\Z^I$ and $[a]_+=\max(a,0)$.

\begin{definition}[\cite{BZ05}, Definition~3.1]
An extended exchange matrix is an integer matrix
$\widetilde B=(b_{ik})_{i\in I,\,k\in I_{\mathrm{ex}}}$ whose principal
part is skew-symmetrizable.  In the skew-symmetric case we draw an arrow
$i\to k$ when $b_{ik}>0$; multiple arrows record the value of $b_{ik}$.
Let $\Lambda=(\lambda_{ij})_{i,j\in I}$ be skew-symmetric.  We call
$(\Lambda,\widetilde B)$ a \emph{compatible pair} if
\begin{equation}
\label{eq:BZ-compatibility-convention}
 \widetilde B^{\mathsf T}\Lambda=(D\;0)
\end{equation}
for a diagonal matrix $D$ with positive entries; the zero block corresponds
to the frozen columns.
\end{definition}

Since the form used below is even, the based quantum torus restricts from
$\Z[q_{\mathrm{BZ}}^{\pm1/2}]$ to the $\Z[\vloop^{\pm1}]$-module
$\mathcal T(\Lambda)=\bigoplus_{c\in\Z^I}\Z[\vloop^{\pm1}]X^c$ with
multiplication
\begin{equation}
\label{eq:abstract-quantum-torus}
 X^cX^d=\vloop^{\frac12 c^{\mathsf T}\Lambda d}X^{c+d},
\end{equation}
so that $X^{e_i}X^{e_j}=\vloop^{\lambda_{ij}}X^{e_j}X^{e_i}$.  For a
non-even form one makes the scalar extension to
$\Z[\vloop^{\pm1/2}]$.  Let $\mathcal F(\Lambda)$ be the skew field of
fractions.

\begin{definition}
A toric frame in a skew field $\mathcal F$ is a map
$M\colon\Z^I\to\mathcal F\setminus\{0\}$ whose image
$\{M(c):c\in\Z^I\}$ is the distinguished basis of an isomorphic copy of
$\mathcal T(\Lambda)$, as in \cite[Def.~4.3]{BZ05}.  Thus
\begin{equation}
\label{eq:toric-frame-product}
 M(c)M(d)=\vloop^{\frac12c^{\mathsf T}\Lambda d}M(c+d).
\end{equation}
A quantum seed is the pair $(M,\widetilde B)$, with
$(\Lambda_M,\widetilde B)$ compatible.  Its variables are $X_i=M(e_i)$.
We write $M_0(c)=X^c$ for the canonical initial toric frame, and
$M_\Sigma$ for the toric frame of a seed $\Sigma$.

For $k\in I_{\mathrm{ex}}$, put
$c_k^\pm=-e_k+\sum_{i\in I}[\pm b_{ik}]_+e_i$.  Mutation replaces $X_k$
by
\begin{equation}
\label{eq:quantum-seed-mutation}
 X_k'=M(c_k^+)+M(c_k^-)
\end{equation}
and fixes the other variables.  The matrix mutates by
\begin{equation}
\label{eq:matrix-mutation}
 b'_{ij}=
 \begin{cases}
  -b_{ij},&i=k\text{ or }j=k,\\
  b_{ij}+[b_{ik}]_+[b_{kj}]_+
  -[-b_{ik}]_+[-b_{kj}]_+,&i,j\ne k.
 \end{cases}
\end{equation}
The full mutated toric frame $M'$ is the one in
\cite[Prop.~4.7 and Def.~4.8]{BZ05}; in particular $M'(e_k)=X_k'$, and
its commutation form is compatible with $\widetilde B'$.
\end{definition}

Following \cite[Def.~4.12]{BZ05}, write
$\mathcal A(M_0,\widetilde B)$, or equivalently
$\mathcal A(\Lambda,\widetilde B)$, for the quantum cluster algebra
generated by all cluster variables together with the frozen variables
and their inverses.  This is the integral $\Z[\vloop^{\pm1}]$-form used
throughout the paper.  After the scalar extension
$\Z[\vloop^{\pm1}]\subset\Z[\vloop^{\pm1/2}]$, it identifies with the
ring denoted $\mathcal A^{\mathrm{loc}}_{(L^{\mathrm{CW}},\widetilde B)}$
in \cite[Sec.~4.4]{CW19}, under $q_{\mathrm{CW}}=\vloop^{-1}$ and
$L^{\mathrm{CW}}=-\Lambda$:
\begin{equation}
\label{eq:CW-half-scalar-extension}
 \mathcal A^{\mathrm{loc}}_{(L^{\mathrm{CW}},\widetilde B)}
 \cong
 \Z[\vloop^{\pm1/2}]
 \mathop{\otimes}_{\Z[\vloop^{\pm1}]}
 \mathcal A(\Lambda,\widetilde B).
\end{equation}
Indeed, mutation $\Lambda'=E_k^{\mathsf T}\Lambda E_k$
\cite[Prop.~4.7]{BZ05} preserves evenness, so all mutated frames are
defined over $R=\Z[\vloop^{\pm1}]$.  Since
$S=\Z[\vloop^{\pm1/2}]$ is free of rank two over $R$, adjoining $\vloop^{1/2}$ gives the $S$-algebra on the same mutation variables.
The upper quantum cluster algebra is
\begin{equation}
\label{eq:upper-quantum-cluster-algebra}
 \mathcal U(M_0,\widetilde B)
 =\bigcap_{\Sigma'\sim\Sigma}
   \mathcal T_{\Sigma'}
\end{equation}
inside the common skew field, where $\mathcal T_{\Sigma'}$ is the
Laurent torus of the seed $\Sigma'$.  At $\vloop=1$, formula
(\ref{eq:quantum-seed-mutation}) becomes the ordinary exchange relation
and (\ref{eq:matrix-mutation}) becomes quiver mutation \cite{FZ02}.

\begin{remark}
\label{rem:CW-sign-convention}
Under the substitutions of (\ref{eq:BZ-CW-source-notation}), the based
quantum torus is unchanged, since
$q_{\mathrm{CW}}^{L^{\mathrm{CW}}_{ij}}=\vloop^{\lambda_{ij}}$, and
skew-symmetry of $\Lambda$ turns (\ref{eq:BZ-compatibility-convention})
into the third identity of (\ref{eq:BZ-CW-source-notation}).  Cautis and Williams use
the quiver convention
$b_{ij}=\#(j\to i)-\#(i\to j)$, the opposite of our drawing convention;
all matrix calculations in this paper use (\ref{eq:matrix-mutation}).
\end{remark}

\subsection{Quantum monoidal seeds}

Let $\mathcal C$ be a graded finite-length monoidal abelian category
with renormalized $r$-matrices.  A simple object $M$ is \emph{real} if
$M\conv M$ is simple.  Two objects $M,N$ \emph{$\vloop$-commute} if
their two convolution orders are isomorphic after a grading shift.
Write $M\{r\}$ for the grading shift normalized by
\begin{equation}
\label{eq:grading-shift-convention}
 [M\{r\}]=\vloop^r[M]\qquad\text{in }K_0(\mathcal C).
\end{equation}
\begin{definition}
\label{def:half-shift-extension}
Let $\mathcal C$ be a monoidal abelian category on which the integral
grading shifts form an exact central monoidal action: there are fixed
coherent isomorphisms
$(M\{a\})\conv N\cong(M\conv N)\{a\}\cong M\conv(N\{a\})$ for $a\in\Z$,
compatible with the associator and the addition of shifts.  We require
the induced half-braiding of $T=\mathbf1\{1\}$ to satisfy
$\beta_T=\operatorname{id}_{T\conv T}$, where
$\beta_X\colon T\conv X\xrightarrow{\sim}X\conv T$ is the central
structure.  Define $\mathcal C^{1/2}$ by adjoining formal half-integer grading
shifts as follows.  Its objects are finite direct
sums of pairs $(M,\epsilon)$ with $M\in\mathcal C$ and
$\epsilon\in\{0,1\}$; morphisms are given by
$\operatorname{Hom}_{\mathcal C^{1/2}}((M,\epsilon),(N,\delta))
=\operatorname{Hom}_{\mathcal C}(M,N)$ when $\epsilon=\delta$ and are
zero otherwise.  Write $(M,0)=M$ and $(M,1)=M\{\frac12\}$.  The
grading shift by one half sends $(M,0)$ to $(M,1)$ and $(M,1)$ to $(M\{1\},0)$, while
\[
 (M,\epsilon)\conv(N,\delta)=
 \bigl((M\conv N)\{\lfloor(\epsilon+\delta)/2\rfloor\},
       \epsilon+\delta\!\pmod 2\bigr).
\]
The coherence isomorphisms of the central grading action on $\mathcal C$
supply the associator and unit constraints.
\end{definition}

\begin{lemma}
\label{lem:CW-half-shift-extension}
For $\mathcal C=\KP^\eta_{\G,N}$, the category $\mathcal C^{1/2}$ is a
locally finite, finite-length, rigid monoidal abelian category, and
\begin{equation}
\label{eq:half-shift-K0}
 K_0(\mathcal C^{1/2})
 \cong
 \Z[\vloop^{\pm1/2}]
 \mathop{\otimes}_{\Z[\vloop^{\pm1}]}K_0(\mathcal C).
\end{equation}
Under $q_{\mathrm{CW}}=\vloop^{-1}$ and $L^{\mathrm{CW}}=-\Lambda$, one
has $[\mathbf1\{\tfrac12\}]=\vloop^{1/2}=q_{\mathrm{CW}}^{-1/2}$, the
convention of \cite[Sec.~4.3]{CW19}.  Moreover, every monoidal mutation
constructed from the seed of this paper lies in the full subcategory
of objects $(M,0)$, identified with $\mathcal C$.
\end{lemma}

\begin{proof}
For the grading line $T=\chi_{\mathrm{rot}}^{-1}$, the half-braiding
$\beta_T$ is the identity. The construction is therefore monoidal, and its
underlying abelian category is $\mathcal C\oplus\mathcal C$.
Finite length and local finiteness follow from
Theorem~\ref{thm:CW-monoidal-structure} and
the local-finiteness argument in Section~\ref{sec:convolution-correspondence};
rigidity follows from
$(M\{\tfrac12\})^R=M^R\{-\tfrac12\}$ and the analogous left-dual
formula. Adjoining the central class $\vloop^{1/2}$ with square
$\vloop$ gives \eqref{eq:half-shift-K0}.

By \cite[Prop.~4.7]{BZ05}, matrix mutation preserves evenness of
$\Lambda$. Hence the grading shifts in normalized monomials, as well as
the exchange shifts $s_k^\pm$ and $d_k/2$, are integral. The full subcategories
with $\epsilon=0$ and $\epsilon=1$ are Serre subcategories, and convolution
adds $\epsilon$ modulo two. If the existing seed objects have $\epsilon=0$,
both outer terms of an exchange sequence have $\epsilon=0$, so its middle
term and its new simple factor do as well.
Induction proves the assertion.
\end{proof}

\begin{definition}
\label{def:homogeneous-r-matrices}
In a category with a central grading action, a homogeneous system of
renormalized $r$-matrices consists of integers $\Lambda(M,N)$ and
nonzero degree-zero morphisms
$r_{M,N}\colon M\conv N\to N\conv M\{\Lambda(M,N)\}$ for nonzero $M,N$,
satisfying \cite[Def.~4.1]{CW19} with each target shifted by its
$r$-matrix degree, as in \cite[Def.~4.20]{CW19} and
\cite[Thm.~7.1]{CW23}. Thus composites are compared after transporting
their shifts through convolution: equality of degrees gives the
stated identity, and the strict inequalities in axioms (iii) and (v)
give zero composites. In axiom (iv) the composite has target shift
$\Lambda(M,N)+\Lambda(N,M)$ and is nonzero exactly when this sum is
zero.
\end{definition}

Suppose $\{M_i\}_{i\in I}$ are pairwise $\vloop$-commuting real simples
whose classes satisfy $[M_i][M_j]=\vloop^{\lambda_{ij}}[M_j][M_i]$.  In
the localized setting, assume in addition that every frozen object
$M_f$, $f\in I_{\mathrm{fr}}$, is invertible.  For
$a=(a_i)\in\Z_{\geq0}^{I_{\mathrm{ex}}}\times\Z^{I_{\mathrm{fr}}}$, form
an ordered list $N_1,\ldots,N_s$ containing $a_i$ copies of $M_i$ when
$a_i\geq0$, and $-a_f$ copies of $M_f^{-1}$ when a frozen exponent $a_f$
is negative.  Following the notation of
\cite[Sec.~4.4, paragraph preceding Eq.~(4.33)]{CW19}, abbreviate the
normalized object by
\begin{equation}
\label{eq:normalized-categorical-monomial}
 \mathbf M(a)=N_1\odot\cdots\odot N_s,
\end{equation}
with the empty product equal to the monoidal unit.  Explicitly, in
the fixed index order,
\[
 \mathbf M(a)=
 \bigl(M_1^{*a_1}*\cdots*M_{|I|}^{*a_{|I|}}\bigr)
 \left\{-\frac12\sum_{i<j}a_i a_j\lambda_{ij}\right\},
\]
where negative powers mean convolution powers of the inverse frozen
object.  Thus $\odot$ denotes an ordinary convolution with the grading shift
prescribed by \eqref{eq:toric-frame-product}.
Its class is the image of $M_\Sigma(a)$ under
$M_\Sigma(e_i)\mapsto[M_i]$.  The $r$-matrix relations identify the
objects obtained from different orders.
For a nonlocalized seed, one instead restricts all exponents to be
nonnegative.  The middle term of a mutation, whose factors need not
$\vloop$-commute, is always written as an ordinary convolution with its
grading shift displayed.

\begin{definition}
\label{def:quantum-monoidal-seed}
The collection $\{M_i\}_{i\in I}$ is a \emph{quantum monoidal cluster}
if the map from the localized quantum affine space defined by
$X^{e_i}\mapsto[M_i]$ and $X^{-e_f}\mapsto[M_f^{-1}]$ is injective.
When grading shifts act freely on simple isomorphism classes, this is
equivalent to the objects $\mathbf M(a)$ being pairwise nonisomorphic
even up to grading shift.  Together with an exchange matrix
$\widetilde B$ compatible with $\Lambda$, it is a \emph{quantum monoidal seed}.

For $k\in I_{\mathrm{ex}}$, set
$b_k^\pm=\sum_i[\pm b_{ik}]_+e_i$.  If $d_k$ is the $k$th diagonal entry
of $D$ in (\ref{eq:BZ-compatibility-convention}), also put
\begin{equation}
\label{eq:mutation-shift-exponents}
 s_k^\pm=\frac12 e_k^{\mathsf T}\Lambda b_k^\pm,
\end{equation}
so that compatibility gives
$s_k^+-s_k^-=\frac12e_k^{\mathsf T}\Lambda\widetilde B e_k=-d_k/2$.
A simple object $M_k'$ is a monoidal mutation of $M_k$ if replacing
$M_k$ by $M_k'$ gives the mutated quantum monoidal cluster and there is
a short exact sequence
\begin{equation}
\label{eq:abstract-monoidal-mutation}
 0\longrightarrow \mathbf M(b_k^+)\{-d_k/2\}
 \longrightarrow (M_k\conv M_k')\{-s_k^-\}
 \longrightarrow \mathbf M(b_k^-)
 \longrightarrow0.
\end{equation}
The shifts follow by multiplying (\ref{eq:quantum-seed-mutation}) on
the left by $X_k=M_\Sigma(e_k)$ and using (\ref{eq:toric-frame-product}):
\begin{equation}
\label{eq:shifted-quantum-exchange}
\begin{aligned}
 X_kX_k'
 &=\vloop^{s_k^+}M_\Sigma(b_k^+)
   +\vloop^{s_k^-}M_\Sigma(b_k^-),\\
 \vloop^{-s_k^-}X_kX_k'
 &=\vloop^{-d_k/2}M_\Sigma(b_k^+)+M_\Sigma(b_k^-).
\end{aligned}
\end{equation}
Thus the Grothendieck relation of (\ref{eq:abstract-monoidal-mutation})
is exactly (\ref{eq:shifted-quantum-exchange}).  The reverse-order sequence supplied
by the $r$-matrix formalism \cite[Rem.~4.29]{CW19} is, in the same
conventions,
\begin{equation}
\label{eq:abstract-reverse-monoidal-mutation}
 0\longrightarrow \mathbf M(b_k^-)\{-d_k/2\}
 \longrightarrow (M_k'\conv M_k)\{s_k^+\}
 \longrightarrow \mathbf M(b_k^+)
 \longrightarrow0.
\end{equation}
Indeed, right multiplication of (\ref{eq:quantum-seed-mutation}) by
$X_k$ gives its Grothendieck relation
$\vloop^{s_k^+}X_k'X_k=M_\Sigma(b_k^+)
+\vloop^{-d_k/2}M_\Sigma(b_k^-)$.
\end{definition}

In the seed of this paper $D=2I_{2n+r-2}$, so the subobject in
(\ref{eq:abstract-monoidal-mutation}) is $\mathbf M(b_k^+)\{-1\}$.

\begin{definition}
\label{def:monoidal-categorification}
The graded finite-length monoidal abelian category $\mathcal C$ is a
\emph{monoidal categorification} of $\mathcal A(\Lambda,\widetilde B)$
if there is an isomorphism of $\Z[\vloop^{\pm1}]$-algebras
$\mathcal A(\Lambda,\widetilde B)\xrightarrow{\sim}K_0(\mathcal C)$
carrying every quantum cluster monomial to the class of a simple object.
This is the integral form of \cite[Def.~4.26]{CW19}; for an even
compatible pair, Definition \ref{def:half-shift-extension} and
(\ref{eq:CW-half-scalar-extension}) give that definition over
$\Z[\vloop^{\pm1/2}]$.
\end{definition}

\begin{lemma}
\label{lem:homogeneous-reality-criterion}
Let $\mathcal C$ be a $\C$-linear locally finite, finite-length graded
monoidal abelian category with biexact convolution, separable triple
products (see \cite[Def.~4.8 and Prop.~4.9]{CW19}), and a homogeneous
system as in Definition \ref{def:homogeneous-r-matrices}.  Suppose
$K_0(\mathcal C)$ is a domain and the coefficient map
$\Z[\vloop^{\pm1}]\to K_0(\mathcal C)$ is injective.  If
$X,Y,M,N$ are simple, there is an exact sequence
$0\to X\to M\conv N\to Y\to0$, and $X\conv N$ and $Y\conv N$ are simple
and not isomorphic even up to grading shift, then $N$ is real.
\end{lemma}

\begin{proof}
We follow \cite[Prop.~4.16]{CW19}, allowing $r_{N,N}$ to have
nonzero degree. Suppose
$N\conv N$ is not simple.  Tensoring the given exact sequence by $N$
shows that $M\conv N\conv N$ has length two.  A composition series of
$N\conv N$ remains a filtration after tensoring by $M$, with nonzero
subquotients: finite length makes every nonzero object class nonzero,
and the domain hypothesis makes the product of two such classes nonzero.
Hence $N\conv N$ also has length two; write a composition series
$0\to S\to N\conv N\to S'\to0$.

Both $M\conv S$ and $M\conv S'$ are simple, since their lengths add to
two.  If $X\conv N\subseteq M\conv S$, separability applied to
$X\subset M\conv N$ and $S\subset N\conv N$ gives a nonzero subobject
$N'\subset N$ with $N'\conv N\subseteq S$.  Simplicity of $N$ forces
$N'=N$, contradicting $S\subsetneq N\conv N$.  Thus $M\conv S$ and
$X\conv N$ are distinct simple subobjects of the length-two object
$M\conv N\conv N$, and projection to the respective quotients gives
\[
 M\conv S\cong Y\conv N,
 \qquad
 M\conv S'\cong X\conv N.
\]
The extension of $S'$ by $S$ cannot split: applying the first isomorphism with $S'$ as
a second simple subobject would identify $X\conv N$ with $Y\conv N$ up
to a shift.  Hence $S$ is the unique simple subobject and $S'$ the
unique simple quotient of $N\conv N$.

Put $d=\Lambda(N,N)$ and consider
$r_{N,N}\colon N\conv N\to(N\conv N)\{d\}$.  If its image has length
one, uniqueness of the simple quotient of the source and the simple
subobject of the target gives $S'\cong S\{d\}$.  If it has length two,
it is an isomorphism; equality of classes gives
$(\vloop^d-1)[N]^2=0$, so $d=0$.  If this degree-zero endomorphism were
a scalar, \cite[Prop.~4.12]{CW19} would already make $N\conv N$ simple.
Otherwise, local finiteness and algebraic closedness provide an
eigenvalue $c$ for which $r_{N,N}-c\,\mathrm{id}$ has nonzero proper
image, and uniqueness gives $S'\cong S$.  In either case
$S'\cong S\{e\}$ for some integer $e$, and the two displayed
isomorphisms imply $X\conv N\cong(Y\conv N)\{e\}$, the required
contradiction.  Thus $N\conv N$ is simple.
\end{proof}

\begin{lemma}
\label{lem:local-mutation-test}
Let $\mathcal C$ be a $\C$-linear locally finite, finite-length rigid
graded monoidal abelian category.  Assume convolution is biexact,
$\operatorname{End}(\mathbf1)=\C$, and $\mathcal C$ carries a compatible
homogeneous system of renormalized $r$-matrices.  Let $\{M_i\}_{i\in I}$
be a quantum monoidal seed such that $K_0(\mathcal C)$ embeds as a
$\Z[\vloop^{\pm1}]$-subalgebra of a division ring in which its toric
frame and the mutation in direction $k$ are defined.  Suppose the $k$th
exchange column is nonzero.  If a simple object $M_k'$ occurs in a short
exact sequence of the form (\ref{eq:abstract-monoidal-mutation}), with
the displayed grading shifts, then replacing $M_k$ by $M_k'$ gives the
mutated quantum monoidal seed.  The reverse-order exchange sequence is
also exact.
\end{lemma}

\begin{proof}
We use \cite[Prop.~4.30]{CW19}, checking its homogeneous form.
Rigidity supplies separable triple products by
\cite[Prop.~4.9]{CW19}. The division-ring embedding makes $K_0(\mathcal C)$ a domain
and the coefficient map $\Z[\vloop^{\pm1}]\to K_0(\mathcal C)$ injective,
so no simple is isomorphic to a nontrivial
shift of itself. For a real simple $T$, the nonzero map $r_{T,T}$
therefore has degree zero and is scalar. The proofs of
\cite[Props.~4.12, 4.14, 4.17--4.18, Cor.~4.15 and Thm.~4.23]{CW19}
now apply with targets shifted as in
Definition~\ref{def:homogeneous-r-matrices}: their kernels, images,
heads and socles are taken in $\mathcal C$, and every comparison
map has degree zero.

The exchange relation identifies $[M'_k]$ with the mutated variable.
By \cite[Thm.~4.23 and Prop.~4.14]{CW19}, it $\vloop$-commutes
with every unchanged seed object and their convolutions are simple.
The objects $\mathbf M(b_k^+)$ and $\mathbf M(b_k^-)$ are real simple by
\cite[Cor.~4.15]{CW19}; their products with $M'_k$ are simple by
\cite[Thm.~4.23 and Prop.~4.14]{CW19}. They remain distinct up to
shift after multiplication by $[M'_k]$, by cancellation in the
ambient division ring and the nonzero exchange column. Thus
Lemma~\ref{lem:homogeneous-reality-criterion} proves that $M'_k$ is real,
as required in \cite[Prop.~4.30]{CW19}. Simplicity of the new monomials
follows from \cite[Cor.~4.15]{CW19}, their independence from
\cite[Prop.~4.7]{BZ05}, and the reverse sequence from
\cite[Rem.~4.29]{CW19} with the shifts of
\eqref{eq:abstract-reverse-monoidal-mutation}.
\end{proof}

Frozen localization is as in \cite[Cor.~4.35]{CW19}.
The exchange exponents $b_k^\pm$ are nonnegative, and convolution
with an invertible frozen object is exact, so arbitrary frozen
powers preserve the preceding conclusions.

\subsection{Quadratic refinements and Laurent charts}
\label{sec:signed-frame}

Let $\mathbb L$ be a finite-rank lattice with an integral skew form
$(\ ,\ )$, and let $\mathbf e=(e_i)_{i\in I_{\mathrm{ex}}}$ be a
basis of a primitive mutable sublattice, as in \cite[Sec.~3.1]{SS19}.
In this subsection $e_i$ denotes a direction in $\mathbb L$.
Let $\mathbb T_+$ and $\mathbb T_{\mathrm{SS}}$ be the based quantum
tori over $\mathbb Z[\vloop^{\pm1}]$ with products
\[
 Z_\lambda Z_\mu=\vloop^{(\lambda,\mu)}Z_{\lambda+\mu},
 \qquad
 Y_\lambda^{\mathrm{SS}}Y_\mu^{\mathrm{SS}}
 =(-\vloop)^{(\lambda,\mu)}Y_{\lambda+\mu}^{\mathrm{SS}}.
\]

\begin{definition}
\label{def:odd-refinement}
A \emph{quadratic refinement} of the mod-two skew form is a function
$R\colon\mathbb L\to\mathbb F_2$ with
\begin{equation}
\label{eq:refinement-polarization}
 R(\lambda+\mu)=R(\lambda)+R(\mu)+(\lambda,\mu)\pmod2
 \qquad(\lambda,\mu\in\mathbb L).
\end{equation}
A refinement is \emph{odd at a seed} $\mathbf e'$ of the mutation
pattern if $R(e'_k)=1$ for every mutable direction $e'_k$ of that
seed.
\end{definition}

Since $(\lambda,\lambda)=0$, equation
(\ref{eq:refinement-polarization}) gives $R(0)=0$ and
$R(-\lambda)=R(\lambda)$, and by induction
\begin{equation}
\label{eq:refinement-multiple}
 R(\lambda+m\mu)
 =R(\lambda)+mR(\mu)+m(\lambda,\mu)
 \pmod 2
 \qquad(m\in\Z),
\end{equation}
using $(\mu,\mu)=0$; negative $m$ follow from $R(-\mu)=R(\mu)$.

\begin{lemma}
\label{lem:odd-refinement-mutation}
Let $\mathbb L$ be a finite-rank lattice with a skew form and mutation
data as in \cite{SS19}, and let $R$ be a quadratic refinement
(\ref{eq:refinement-polarization}) of its mod-two skew form which is
odd at a seed $\mathbf e$.  Then $R$ is odd at every seed obtained from
$\mathbf e$ by a finite sequence of signed lattice mutations
\[
 \mu_k^\epsilon(e_k)=-e_k,
 \qquad
 \mu_k^\epsilon(e_i)
 =e_i+[\epsilon(e_i,e_k)]_+e_k\quad(i\ne k).
\]
\end{lemma}

\begin{proof}
It suffices to treat one mutation.  First,
$R(-e_k)=R(e_k)=1$.  Next fix $i\ne k$ and put
$m=[\epsilon(e_i,e_k)]_+\geq0$.  By
(\ref{eq:refinement-multiple}) and $R(e_k)=1$,
\[
 R\bigl(e_i+me_k\bigr)
 =R(e_i)+m\bigl(1+(e_i,e_k)\bigr)\pmod2 .
\]
For $m=0$ the correction vanishes.  Otherwise
$m\equiv(e_i,e_k)\pmod2$, so $m(1+(e_i,e_k))$ is even.  Thus
$R(e_i+me_k)=1$.
\end{proof}

For $\mathbb T_{\mathrm{SS}}$, put $q=\vloop^{-1}$ and
\[
 \Psi_q(X)=\prod_{m\geq0}(1-q^{2m+1}X)^{-1},
 \qquad \Psi_q(q^2X)=(1-qX)\Psi_q(X),
\]
as in \cite[Eqs.~(2.10)--(2.11)]{SS19}, and use
$\operatorname{Ad}_g(a)=gag^{-1}$.  Define
\begin{equation}
\label{eq:SS-chart-pullback}
 \theta_k^{\mathrm{SS}}
 :=\operatorname{Ad}_{\Psi_q(Y^{\mathrm{SS}}_{-e_k})},
 \qquad
 \mathbb T^{\mathrm{SS}}_{\mu_k(\mathbf e)}
 :=(\theta_k^{\mathrm{SS}})^{-1}(\mathbb T_{\mathrm{SS}}).
\end{equation}
The difference equation makes conjugation on every lattice monomial
a finite product of rational binomial factors, so
$\theta_k^{\mathrm{SS}}$ is an automorphism of the fraction field.
The chart embedding is its inverse: $F$ belongs to this chart precisely
when $\theta_k^{\mathrm{SS}}(F)$ is Laurent in the initial torus.
Write
\begin{equation}
\label{eq:SS-Laurent-bound}
 \begin{split}
 \mathcal L^{\mathrm{SS}}_{\mathbf e}
 &:=\mathbb T_{\mathrm{SS}}\cap
       \bigcap_{k\in I_{\mathrm{ex}}}(\theta_k^{\mathrm{SS}})^{-1}
                          (\mathbb T_{\mathrm{SS}})\\
 &=\{F\in\mathbb T_{\mathrm{SS}}:
       \theta_k^{\mathrm{SS}}(F)\in\mathbb T_{\mathrm{SS}}
       \text{ for every mutable }k\}.
 \end{split}
\end{equation}
These are the Laurent conditions in \cite[proof of Lem.~4.2 and
Sec.~5.4]{SS19}.  The displayed chart definition following Eq.~(3.4)
of that version uses the forward image under $\operatorname{Ad}$;
with the above convention its proofs instead give the inverse image
in \eqref{eq:SS-chart-pullback}.

\begin{lemma}
\label{lem:dilog-transport}
Let $\mathbb L$ have an integral skew form and a seed $\mathbf e$,
and let $R$ be a quadratic refinement odd at $\mathbf e$.  Let
$\mathbb T_+$ and $\mathbb T_{\mathrm{SS}}$ be the based quantum tori with
parameters $\vloop$ and $-\vloop$, respectively, and set
$\rho(Z_\lambda)=(-1)^{R(\lambda)}Y^{\mathrm{SS}}_\lambda$.
For $q=\vloop^{-1}$ define
\[
 \theta_k^+
 :=\operatorname{Ad}_{\Psi_q(-Z_{-e_k})},
 \qquad
 \mathbb T^+_{\mu_k(\mathbf e)}
 :=(\theta_k^+)^{-1}(\mathbb T_+),
 \qquad
 \mathcal L^+_{\mathbf e}
 :=\mathbb T_+\cap\bigcap_k\mathbb T^+_{\mu_k(\mathbf e)}.
\]
Then $\rho$ intertwines $\theta_k^+$ with
$\theta_k^{\mathrm{SS}}$ for each $k$, identifies their inverse-image
charts, and restricts to an isomorphism
\begin{equation}
\label{eq:signed-upper-transport}
 \rho\colon\mathcal L^+_{\mathbf e}
 \xrightarrow{\ \sim\ }\mathcal L^{\mathrm{SS}}_{\mathbf e}.
\end{equation}
The same assertions hold at every seed reached by signed lattice
mutations, with its corresponding directions and based quantum tori.
\end{lemma}

\begin{proof}
The quadratic-refinement identity makes $\rho$ an algebra isomorphism,
because its polarization cancels the sign in the torus product.  Since $R(-e_k)=R(e_k)=1$,
\[
 \rho(-Z_{-e_k})=Y^{\mathrm{SS}}_{-e_k}.
\]
The scalar parameter $q=\vloop^{-1}$ is unchanged.  Applying $\rho$
to the finite binomial ratios obtained from the difference equation
for $\Psi_q$ therefore gives
\[
 \rho\circ\theta_k^+=\theta_k^{\mathrm{SS}}\circ\rho,
 \qquad
 \rho\bigl((\theta_k^+)^{-1}(\mathbb T_+)\bigr)
 =(\theta_k^{\mathrm{SS}})^{-1}(\mathbb T_{\mathrm{SS}}).
\]
Taking the finite intersections proves
(\ref{eq:signed-upper-transport}).  Lemma
\ref{lem:odd-refinement-mutation} preserves the needed oddness at
every subsequent seed, so the same calculation applies there.
\end{proof}

\begin{proposition}
\label{prop:v-upper-A-comparison}
Let $\mathbb L$, $\mathbf e$, and $R$ be as in Lemma
\ref{lem:dilog-transport}, and let $\boldsymbol\alpha$ be an
integral basis of $\mathbb L$ satisfying
$(e_i,\alpha_j)=\delta_{ij}$ for mutable $i$.  Put
$\Lambda_{ij}=2(\alpha_i,\alpha_j)$ and
$e_k=\sum_j b_{jk}\alpha_j$.  With the frozen variables inverted,
\[
 \mathcal L^+_{\mathbf e}
 =\mathcal U(\Lambda,\widetilde B)
 \quad\text{inside }\operatorname{Frac}(\mathbb T_+).
\]
\end{proposition}

\begin{proof}
We identify each of the once-mutated tori.  Fix $k$ and write
\[
 c=-\alpha_k+\sum_{j\ne k}[b_{jk}]_+\alpha_j,
 \qquad f=-e_k.
\]
Replacing $\alpha_k$ by $c$ is a unimodular basis change, and
$(f,c)=1$, while $(f,\alpha_j)=0$ for $j\ne k$.  Put $X=Z_f$.
The torus product gives $XZ_c=q^{-2}Z_cX$, and the difference
identity for $\Psi_q$ gives
\[
 (\theta_k^+)^{-1}(Z_c)
 =Z_c(1+q^{-1}X)=Z_c+Z_{c+f}.
\]
Here $Z_cX=qZ_{c+f}$ and
\[
 c+f=-\alpha_k+\sum_{j\ne k}[-b_{jk}]_+\alpha_j.
\]
Thus the displayed sum is exactly the exchange sum in
\eqref{eq:quantum-seed-mutation}, with coefficient one on both monomials.  Every
$Z_{\alpha_j}$ for $j\ne k$ is fixed.  The images of the basis
monomials and their inverses generate
$(\theta_k^+)^{-1}(\mathbb T_+)$, so this is the full based quantum torus of
the seed mutated at $k$, including the frozen inverses.

By \cite[Thm.~5.1 and Eq.~(5.3)]{BZ05}, this finite upper bound
is the intersection over all seeds. Apply that result with form
$\Lambda/2$ and parameter $t=\vloop^2$: compatibility becomes
$\widetilde B^{\mathsf T}(\Lambda/2)=(I\;0)$, and
$\mathbb Z[t^{\pm1/2}]=\mathbb Z[\vloop^{\pm1}]$.
The based product is
$t^{(\alpha_i,\alpha_j)/2}=\vloop^{(\alpha_i,\alpha_j)}$,
so the torus and exchange monomials are exactly those above.
\end{proof}

\subsection{Iterated quantum monoidal mutation}
\label{sec:KKKO-comparison}

The construction of successive monoidal mutations follows \cite[Thm.~4.34(ii)]{CW19}, with the
factorization construction of homogeneous $r$-matrices from
\cite[Sec.~7]{CW23}.  The analogous admissibility theorem for symmetric
KLR algebras is \cite[Thm.~7.1.3 and Cor.~7.1.4]{KKKO18}.

\begin{lemma}
\label{lem:frozen-localized-factorization}
Let $\mathcal C$ be a $\C$-linear, locally finite, finite-length graded
monoidal abelian category with biexact product and
$\operatorname{End}(\mathbf1)=\C$, whose exact grading shift is a central
monoidal action.  Suppose that $\mathcal C$ has separable triple products
and homogeneous renormalized $r$-matrices
\[
 r_{M,N}:M\conv N\longrightarrow N\conv M\{\Lambda(M,N)\}
\]
which are nonzero for nonzero $M,N$ and satisfy Definition
\ref{def:homogeneous-r-matrices}.

Let a quantum monoidal seed in $\mathcal C$ have tensor-invertible frozen
objects, admit a mutation in every exchangeable direction, and have no zero
mutable column after any sequence of matrix mutations.  Assume also that the
compatibility diagonal is $D=2I_{I_{\mathrm{ex}}}$ and that the commutation
form of every seed obtained by matrix mutation is even, so that every based monomial
normalization and exchange shift is integral.  Suppose its initial
map from the quantum affine space is injective, $K_0(\mathcal C)$ is a domain, and
the coefficient map $R=\Z[\vloop^{\pm1}]\to K_0(\mathcal C)$ is injective.
Finally, suppose that the initial map extends, in the same toric frame, to
an injective cluster-algebra map after extension of scalars
\[
 \mathbb Q(\vloop)\otimes_R\mathcal A(\Lambda,\widetilde B)
 \lhook\joinrel\longrightarrow
 \mathbb Q(\vloop)\otimes_RK_0(\mathcal C).
\]
Then the seed admits every finite sequence of quantum monoidal mutations
and the corresponding forward and reverse exchange sequences.  Every
quantum cluster variable is represented by a real simple object and every
localized quantum cluster monomial by a simple normalized product.
\end{lemma}

\begin{proof}
We apply the proof of \cite[Thm.~4.34(ii)]{CW19} in $\mathcal C$. That theorem is stated for an ungraded category and its
graded realization; here the following checks replace passage to
the ungraded category.

The domain and coefficient hypotheses make grading shifts free on
simple isomorphism classes. Finite length then makes representatives
of their shift orbits an $R$-basis of $K_0(\mathcal C)$.
As in the proof of Lemma~\ref{lem:local-mutation-test}, the
results on heads, socles, and commutation in
\cite[Props.~4.12, 4.14, 4.17--4.18, Cor.~4.15 and Thm.~4.23]{CW19}
hold when the target of each $r_{M,N}$ is shifted by $\{\Lambda(M,N)\}$.
The same argument applies to \cite[Lem.~4.24]{CW19}. For reality we use
Lemma~\ref{lem:homogeneous-reality-criterion}.

It suffices to show that, after a first mutation $x$, the seed still
admits mutation in every mutable direction.
Write $\widetilde B'=\mu_x(\widetilde B)$ and consider a second
direction $y$ with $b=b_{xy}>0$; the other cases are treated in
\cite[proof of Thm.~4.34(ii)]{CW19} by interchanging the directions
or commuting the mutations. For $i\ne x$, put
\[
 p_i=[b_{iy}]_+,\qquad q_i=[-b'_{iy}]_+,\qquad
 a_i=[-b_{iy}]_+-q_i,\qquad c_i=b[b_{ix}]_+-a_i,
\]
and set the $x$-coordinates of these vectors to zero. Matrix mutation
gives $a_i,c_i\geq0$. Using the normalized seed monomials of
\eqref{eq:normalized-categorical-monomial}, define
\[
 P=\mathbf M(p),\qquad Q=\mathbf M(q),\qquad
 A=\mathbf M(a),\qquad B=\mathbf M(c),\qquad L=(M'_x)^{*b}.
\]
Then $b_y^+=be_x+p$, $b_y^-=a+q$, and $b\,b_x^+=a+c$.
Thus the existing exchange sequence at $y$ has terms
\[
 E_+=(M_x^{*b}\nconv P)\{-1\},\qquad
 E_0=(M_y\conv M'_y)\{-s_y^-\},\qquad E_-=Q\nconv A.
\]
The degree calculation preceding \cite[Eq.~(4.40)]{CW19} gives
\[
 \delta=\Lambda(L,E_+)=\Lambda(L,E_0)=\Lambda(L,E_-).
\]
Convolving that sequence with $L$ on each side, and shifting the
second row by $\{\delta\}$, gives the diagram
\[
\begin{tikzcd}[column sep=small,row sep=large]
0\arrow[r]
 &L\conv E_+\arrow[r]\arrow[d,"r_{L,E_+}"']
 &L\conv E_0\arrow[r]\arrow[d,"r_{L,E_0}"]
 &L\conv E_-\arrow[r]\arrow[d,"r_{L,E_-}","\sim"']
 &0\\
0\arrow[r]
 &(E_+\conv L)\{\delta\}\arrow[r]
 &(E_0\conv L)\{\delta\}\arrow[r]
 &(E_-\conv L)\{\delta\}\arrow[r]
 &0.
\end{tikzcd}
\]
Both rows are exact, and all vertical maps have degree zero.
The naturality axiom \cite[Def.~4.1(v)]{CW19}, together with the equality
of the displayed $r$-matrix degrees, makes the diagram commute. Its right vertical map is an isomorphism,
since $E_-$ is a monomial in unchanged seed objects. Taking images
therefore gives a short exact sequence in $\mathcal C$.
As in the cited proof, its middle term is a shift of $M_y\conv H$,
where $H=\operatorname{hd}(L\conv M'_y)$ is simple. Its outer terms
are shifts of $X\conv A$ and $Y\conv A$, where
$X=P\nconv B$ and $Y=L\nconv Q$.
The shift calculation following that diagram, using $D=2I$, gives
the relative coefficient $\vloop^{-1}$ between the outer terms.
All shifts are integral because the commutation forms are even.

Let $\Phi$ be the cluster variable obtained by mutation first at $x$
and then at $y$. Using the assumed embedding after extension of scalars
to $\mathbb Q(\vloop)$, we make the comparison in the cited proof with
the short exact sequence obtained by taking images of the vertical
$r$-matrices:
\begin{equation}
\label{eq:graded-propagation-cancellation}
 \vloop^{m_1}[M_y][H]=\vloop^{m_2}[M_y]\Phi[A]
 \qquad(m_1,m_2\in\mathbb Z).
\end{equation}
Cancel $[M_y]$ in the domain and expand $\Phi$ in the basis given by
classes of simple objects, choosing one representative from each orbit
under grading shifts. Since $\Phi$ $\vloop$-commutes with $[A]$,
\cite[Thm.~4.23 and Prop.~4.14]{CW19} make the product of $A$
with each occurring simple object simple. These products belong
to distinct shift orbits by cancellation of $[A]$. As $H$ is
simple, the equality forces $\Phi=[M''_y]$ for a single simple
object with an integral shift absorbed into its normalization.

For integers $m,\ell$, the two ordered identities at the end of
the cited proof become
\begin{equation}
\label{eq:graded-propagation-ordered-identities}
\begin{aligned}
 \vloop^{-m}[M_y][M''_y]&=\vloop^{-1}[X]+[Y],\\
 \vloop^{-\ell}[M''_y][M_y]&=[X]+\vloop^{-1}[Y].
\end{aligned}
\end{equation}
The argument identifying heads and socles in the cited proof applies to
$M_y\conv M''_y$ and $M''_y\conv M_y$. If $Y$ were the socle of $(M_y\conv M''_y)\{-m\}$, it would give
$\Lambda(M''_y,M_y)+\Lambda(M_y,M''_y)=-2$, contrary to
$r$-matrix axiom (iv). Thus
\begin{equation}
\label{eq:graded-propagated-exchange-sequence}
 0\longrightarrow X\{-1\}
 \longrightarrow(M_y\conv M''_y)\{-m\}
 \longrightarrow Y\longrightarrow0.
\end{equation}
The unchanged seed objects are real and their classes
$\vloop$-commute with $\Phi$, so
\cite[Thm.~4.23 and Prop.~4.14]{CW19} give the required
commutations and simple products with $M''_y$, including those
with the real simple monomials $X,Y$. The products
$X\conv M''_y$ and $Y\conv M''_y$ are distinct up to shift by
the nonzero exchange column and cancellation. Apply
Lemma~\ref{lem:homogeneous-reality-criterion} to
\eqref{eq:graded-propagated-exchange-sequence}, shifted by $\{m\}$,
to obtain reality
of $M''_y$. The remaining conclusions of the induction are
\cite[Cor.~4.15, Rem.~4.29 and Thm.~4.34(ii)]{CW19}; independence
follows from the mutated toric frame. Finally, convolution with
invertible frozen objects preserves simplicity, giving the
localized assertion as in \cite[Cor.~4.35]{CW19}.
\end{proof}

\section{The group and its lattice models}
\label{sec:group}

Throughout, let $n\geq2$ and $r\geq1$.

\subsection{The quotient group and the framed quiver}

Put $V=\C^n$ and $N_{n,r}=V^{\oplus r}$. The homomorphism
\[
 GL_n\times(\C^\times)^r\longrightarrow GL(N_{n,r}),\qquad
 (g,z_1,\ldots,z_r)\longmapsto(gz_1,\ldots,gz_r)
\]
has central kernel $K_r=\{(z^{-1}I_n,z,\ldots,z)\}$. Its image is
\begin{equation}
\label{eq:gln-effective-group}
 G_{n,r}=\frac{GL_n\times(\C^\times)^r}{K_r}
       \simeq GL_n\times(\C^\times)^{r-1}.
\end{equation}
This quotient acts faithfully on $N_{n,r}$. In the coordinates
$A=gz_1$ and $h_a=z_1z_a^{-1}$, the action is
\[
 (A,h_2,\ldots,h_r)(v_1,\ldots,v_r)
       =(Av_1,Ah_2^{-1}v_2,\ldots,Ah_r^{-1}v_r).
\]
Equivalently, setting $B=A$ and $c_a=h_a^{-1}$ gives
\begin{equation}
\label{eq:hf-effective-group}
 (B,c_2,\ldots,c_r)(v_1,\ldots,v_r)
       =(Bv_1,c_2Bv_2,\ldots,c_rBv_r).
\end{equation}
The central cocharacter $\eta(z)=(zI_n,1,\ldots,1)$ in these
coordinates acts with weight one on every summand of $N_{n,r}$.

This is the natural gauge-group action on the representation space of the
framed star-shaped quiver with arrows
\[
 f\longrightarrow i\longleftarrow j_a\qquad(2\leq a\leq r),
 \qquad \dim V_i=n,\quad\dim V_{j_a}=\dim W_f=1.
\]
Its gauge group is $G_Q=GL(V_i)\times\prod_{a=2}^rGL(V_{j_a})$,
its flavor group is $F_Q=GL(W_f)$, and its polarization is
$\Hom(W_f,V_i)\oplus\bigoplus_{a=2}^r\Hom(V_{j_a},V_i)$.
For gauge coordinates $(g,h_2,\ldots,h_r)$ and flavor coordinate $f$,
the change of variables $B=gf^{-1}$ and $c_a=fh_a^{-1}$ gives
\begin{equation}
\label{eq:hf-flavor-splitting}
 G_Q\times F_Q\simeq G_{n,r}\times\C_f^\times,
\end{equation}
where the last factor acts trivially on the representation. Thus
the gauge datum is $(G_{n,r},N_{n,r})$ on the slice $f=1$; the extra
flavor character may be specialized to one in the equivariant theory.

\subsection{Coweights and connected components}

Let $T_{n,r}$ be the diagonal maximal torus. In the quotient presentation
write a coweight lift as
$(\mu_1,\ldots,\mu_n;m_1,\ldots,m_r)\in\Z^{n+r}$, and put
$\kappa=(-1,\ldots,-1;1,\ldots,1)$. Then
\begin{equation}
\label{eq:coweight-and-component-lattices}
 X_*(T_{n,r})=\Z^{n+r}/\Z\kappa,
 \qquad
 \pi_1(G_{n,r})=X_*(T_{n,r})/
       \langle e_i-e_{i+1}:1\leq i<n\rangle.
\end{equation}
Set $\omega_k=(1^k,0^{n-k};0,\ldots,0)$, $1\leq k\leq n$,
and $\gamma_a=(0,\ldots,0;e_a)$. The distinction between these
two quotients gives
\begin{equation}
\label{eq:gln-component-relation}
 \omega_n=\sum_{a=1}^r\gamma_a\quad\text{in }X_*(T_{n,r}),
 \qquad
 n[\omega_1]=\sum_{a=1}^r[\gamma_a]\quad\text{in }\pi_1(G_{n,r}),
\end{equation}
and hence
\begin{equation}
\label{eq:original-component-relation}
 \pi_1(G_{n,r})\simeq
 \bigl\langle[\omega_1],[\gamma_1],\ldots,[\gamma_r]
       \bigm|n[\omega_1]=\textstyle\sum_a[\gamma_a]\bigr\rangle
 \simeq\Z^r.
\end{equation}
Arithmetic in a Schubert subscript takes place in $X_*(T_{n,r})$;
brackets denote the image in the component group. In particular,
$n\omega_1$ and $\omega_n$ index different Schubert varieties,
although they have the same component class.

A dominant lift satisfies $\mu_1\geq\cdots\geq\mu_n$. Write
$\Gr_\lambda^\circ$ for its $G_{n,r}(\O)$-orbit and
$\Gr_\lambda=\overline{\Gr_\lambda^\circ}$. The usual orbit
parametrization and dimension formula \cite[Secs.~1.1--1.2]{Zhu17} give
\begin{equation}
\label{eq:schubert-dimension-component}
 \dim\Gr_\lambda=\sum_{i=1}^n(n+1-2i)\mu_i,
 \qquad
 [\Gr_\lambda]=\Bigl(\sum_i\mu_i\Bigr)[\omega_1]
                       +\sum_a m_a[\gamma_a].
\end{equation}
Both expressions are unchanged on adding $\kappa$. The valuation
of $\det A$ is $\sum_i\mu_i+nm_1$.
Each $\omega_k$ is minuscule, with
$\Gr_{\omega_k}\simeq\operatorname{Gr}(n-k,V)$ for $k<n$;
$\omega_n$ and all $\gamma_a$ have point orbits. A nondominant
subscript means its dominant Weyl representative.

\subsection{The lattice model and the derived intersection}
\label{sec:lattice-models}

Set $L_0=V[[t]]$. A point of the affine Grassmannian is represented
by a lattice $\Lambda\subset V((t))$ and fractional ideals
$L_1,\ldots,L_r\subset\K$, modulo
\begin{equation}
\label{eq:lattice-scaling-equivalence}
 [\Lambda,L_1,\ldots,L_r]\sim[t^{-1}\Lambda,tL_1,\ldots,tL_r].
\end{equation}
Indeed, $H^1(\operatorname{Spec}\K,\G_m)=
H^1(\operatorname{Spec}\O,\G_m)=0$, so the central quotient
\eqref{eq:gln-effective-group} is surjective on both loop and
positive-loop points. The remaining ambiguity is
$K_r(\K)/K_r(\O)\simeq\Z$, acting as displayed.
The representation sends this point to the lattice
\[
 \mathbf L=\bigoplus_{a=1}^rL_a\Lambda\subset N_{n,r}((t)),
\]
which is invariant under \eqref{eq:lattice-scaling-equivalence}.

Over the Schubert variety $\Gr_\lambda$ let $\T_\lambda$ be this
tautological lattice bundle and put
$\mathcal E_\lambda=\Gr_\lambda\times N_{n,r}[[t]]$.
The derived restriction of the space of triples and its reduced
classical locus are
\begin{equation}
\label{eq:component-derived-intersection}
 \R_\lambda=
 \mathcal E_\lambda\mathop{\times}^{\mathbf R}_{
             \Gr_\lambda\times N_{n,r}((t))}\T_\lambda,
 \qquad
 \R_\lambda^{\cl}=t_0(\R_\lambda)_{\mathrm{red}}.
\end{equation}
Thus the fiber of $\R_\lambda^{\cl}\to\Gr_\lambda$ at $[\Lambda,L_1,\ldots,L_r]$ is
$\bigoplus_a(L_a\Lambda\cap L_0)$, as in
\cite[Sec.~2(ii)]{BFN18}. The derived construction is made on finite
lattice quotients and passed to the compatible limit
\cite[Sec.~4.1]{CW23}; on every such quotient its structure sheaf is
\begin{equation}
\label{eq:component-koszul-resolution}
 \O_{\R_\lambda}\simeq
 \O_{\mathcal E_\lambda}\mathop{\otimes}^{\mathbf L}_{
       \O_{\Gr_\lambda\times N_{n,r}((t))}}\O_{\T_\lambda}.
\end{equation}
This is the Koszul complex for the difference of the two bundle
inclusions. Even when the ordinary intersection is
$\R_\lambda^{\cl}=\T_\lambda$, its excess Tor terms need not
vanish.

\subsection{Classical loci over minuscule orbits}

For $1\leq k<n$ and $S\subseteq\{1,\ldots,r\}$ set
$\lambda_{k,S}=\omega_k-\sum_{a\in S}\gamma_a$, and put
$\nu_k=(0^{n-k},-1^k;0,\ldots,0)$. The classical loci $\R_\lambda^{\cl}$ are
described by the following table; the last column gives the $a$th
summand of the fiber of $\R_\lambda^{\cl}\to\Gr_\lambda$.
\[
\begin{array}{c|c|c|c}
 \lambda&\text{condition on }\Lambda=L&L_a&
                      a\text{th summand}\\ \hline
 \omega_k&tL_0\subset L\subset L_0, \dim(L_0/L)=k&\O&L\\
 \nu_k&L_0\subset L\subset t^{-1}L_0, \dim(L/L_0)=k&\O&L_0\\
 \gamma_b&L=L_0&t^{\delta_{ab}}\O&t^{\delta_{ab}}L_0\\
 -\gamma_b&L=L_0&t^{-\delta_{ab}}\O&L_0\\
 \lambda_{k,S}&tL_0\subset L\subset L_0, \dim(L_0/L)=k&
 t^{-\mathbf1_{a\in S}}\O&
 \begin{cases}L_0&a\in S,\\L&a\notin S.\end{cases}
\end{array}
\]
All these loci are bundles of constant rank after a finite lattice
quotient. In particular their ordinary intersections are reduced.
For the last row one may equivalently use $L^-=t^{-1}L$ and the
ideals $\O$ on $S$ and $t\O$ off $S$. Then
$L_0\subset L^-\subset t^{-1}L_0$, with
$\dim(L^-/L_0)=n-k$; the $a$th summand of the fiber is $L_0$ for $a\in S$ and
$tL^-$ for $a\notin S$.

The exchange sequences \eqref{eq:cat-outer-first} and
\eqref{eq:cat-outer-other} use $\lambda_{n-1,\{a\}}$.
Its $GL_n$ Schubert variety is $\Gr_{\omega_{n-1}}$, and the
equivalent lattice $L^-$ satisfies $\dim(L^-/L_0)=1$.

\section{The Compatible Pair and the Initial Toric Frame}
\label{sec:compatible-pair}
\label{sec:gln-seed}

Fix $n\geq2$ and $r\geq1$, and use the coordinates
$(A,h_2,\ldots,h_r)$ on $G_{n,r}$ from Section~\ref{sec:group}.
The framed star quiver has central gauge dimension $n$, outer gauge
dimensions one, and one incoming framing of dimension one.
We construct its integral compatible pair and identify the resulting
upper quantum cluster algebra with the Coulomb-branch algebra.

\subsection{The lattice and parameters}
\label{sec:SS-input}

The construction of \cite[Eqs.~(5.1) and (5.6)]{SS19} gives the
following lattice for $G_{n,r}$:
\[
 \mathbb L_{n,r}=
 \bigoplus_{i=1}^n(\mathbb Zp_i\oplus\mathbb Zx_i)
 \oplus\bigoplus_{a=2}^r
          (\mathbb Z\bar p_a\oplus\mathbb Z\bar x_a),
\]
with $(p_i,x_j)=\delta_{ij}$ and
$(\bar p_a,\bar x_b)=\delta_{ab}$; all other independent pairings
are zero.  Here $p$ records a character and $x$ a coweight.
The full flavor-equivariant lattice is
$\mathbb L_{n,r}^{\mathrm{full}}=\mathbb L_{n,r}\oplus\mathbb Z\zeta_f$,
where $\zeta_f$ is the framing character and lies in the radical.
There are no cycle parameters because the underlying graph is a tree.
Equivalently, if $P_Q$ is the weight lattice of the diagonal gauge
torus, then $\mathbb L_{n,r}=P_Q\oplus P_Q$, with
\begin{equation}
\label{eq:SS-skew-form}
 (p_\lambda,x_\mu)=\langle\lambda,\mu\rangle
 =-(x_\mu,p_\lambda),\qquad
 (p_\lambda,p_\mu)=(x_\lambda,x_\mu)=0.
\end{equation}
The dot product identifies weights and coweights in the chosen
$GL$ coordinates.

The scalar conventions are summarized in Table \ref{tab:comparison-parameters};
$\chi_{\mathrm{rot}}$ is the character of the double rotation cover fixed in
Section \ref{sec:loop-rotation-cover}.
\begin{table}[htb]
\centering
\begin{tabular}{@{}lll@{}}
parameter & role & identification\\
\hline
$\chi_{\mathrm{rot}}$ & geometric rotation character & $\vloop^{-1}$\\
$q_{\mathrm{SS}}$ & parameter in the quantum dilogarithm & $\chi_{\mathrm{rot}}$\\
$u$ & factor in the torus multiplication & $(-q_{\mathrm{SS}})^{-1}=-\vloop$\\
$q_{\mathrm{BZ}}$ & torus parameter in \cite{BZ05} & $\vloop$
\end{tabular}
\caption{Parameters in the comparison of quantum tori.}
\label{tab:comparison-parameters}
\end{table}

The Schrader--Shapiro torus is the free
$\mathbb Z[v^{\pm1}]$-module with basis $Y_\lambda^{\mathrm{SS}}$,
for $\lambda\in\mathbb L_{n,r}^{\mathrm{full}}$, and multiplication
\[
 Y_\lambda^{\mathrm{SS}}Y_\mu^{\mathrm{SS}}
 =(-v)^{(\lambda,\mu)}Y_{\lambda+\mu}^{\mathrm{SS}}.
\]
This is \cite[Eq.~(3.1)]{SS19} after
$q_{\mathrm{SS}}=v^{-1}$.  For
$\lambda=\sum_i(a_ip_i+b_ix_i)+
\sum_{a=2}^r(c_a\bar p_a+d_a\bar x_a)+c_f\zeta_f$, put
\begin{equation}
\label{eq:SS-signed-basis}
 Q_{n,r}(\lambda)=\sum_i a_ib_i+\sum_{a=2}^r c_ad_a\pmod2,
 \qquad Y_\lambda=(-1)^{Q_{n,r}(\lambda)}Y_\lambda^{\mathrm{SS}}.
\end{equation}
Polarization gives the multiplication law
\begin{equation}
\label{eq:SS-parameter-change}
 Y_\lambda Y_\mu=v^{(\lambda,\mu)}Y_{\lambda+\mu}.
\end{equation}
A quadratic refinement intertwines mutation only when it is odd
on the mutable directions, as in Lemma~\ref{lem:dilog-transport}.
The refinement $Q_{n,r}$ corrects the initial torus product; the
refinement used in Proposition~\ref{prop:gln-simple-frame} is specified below.

\subsection{The integral compatible pair}

Order the mutable indices as
$u_1,\ldots,u_r,y_1,\ldots,y_{n-1},z_1,\ldots,z_{n-1}$ and the
frozen indices as $d,h_2,\ldots,h_r$.  Define
\begin{equation}
\label{eq:gln-A-basis}
\begin{aligned}
 \alpha_{u_1}&=\sum_{i=1}^nx_i+\sum_{a=2}^r\bar x_a,&
 \alpha_{u_a}&=-\bar x_a &&(a\geq2),\\
 \alpha_{y_k}&=\sum_{i=1}^k(p_i+x_i),&
 \alpha_{z_k}&=\sum_{i=1}^kx_i &&(1\leq k<n),\\
 \alpha_d&=\sum_{i=1}^np_i,&
 \alpha_{h_a}&=\bar p_a &&(a\geq2).
\end{aligned}
\end{equation}
Use the boundary notation
\begin{equation}
\label{eq:gln-boundary-vectors}
 \alpha_{y_0}=\alpha_{z_0}=\alpha_{h_1}=0,\qquad
 \alpha_{z_n}=\sum_a\alpha_{u_a},\qquad
 \alpha_{y_n}=\alpha_d+\alpha_{z_n}.
\end{equation}

\begin{theorem}
\label{thm:gln-compatible-pair}
The vectors \eqref{eq:gln-A-basis} form an integral basis of
$\mathbb L_{n,r}$.  Let $\widetilde B_{n,r}=(b_{ji})$ be the matrix
whose columns are the coordinates in this basis of
\begin{equation}
\label{eq:gln-exchange-vectors}
\begin{aligned}
 e_{u_a}&=\alpha_d+\alpha_{z_{n-1}}-\alpha_{y_{n-1}}
                                      -\alpha_{h_a},\\
 e_{y_k}&=\alpha_{z_{k-1}}-2\alpha_{z_k}+\alpha_{z_{k+1}},\\
 e_{z_k}&=2\alpha_{y_k}-\alpha_{y_{k-1}}-\alpha_{y_{k+1}}.
\end{aligned}
\end{equation}
Then $\Lambda_{n,r}=2((\alpha_i,\alpha_j))$ is even,
the principal part of $\widetilde B_{n,r}$ is skew-symmetric, and
\begin{equation}
\label{eq:gln-compatibility}
 \widetilde B_{n,r}^{\mathsf T}\Lambda_{n,r}
      =\bigl(2I_{2n+r-2}\ \ 0\bigr).
\end{equation}
The mutable directions span a primitive sublattice.
\end{theorem}

\begin{proof}
The inverse change of basis is integral:
\begin{align*}
 x_i&=\alpha_{z_i}-\alpha_{z_{i-1}},&
 p_i&=(\alpha_{y_i}-\alpha_{z_i})
          -(\alpha_{y_{i-1}}-\alpha_{z_{i-1}}),\\
 \bar x_a&=-\alpha_{u_a},& \bar p_a&=\alpha_{h_a}.
\end{align*}
The boundary conventions apply also when $i=n$.  In Darboux
coordinates the exchange directions become
\begin{equation}
\label{eq:gln-Darboux-directions}
\begin{aligned}
 e_{u_1}&=p_n,& e_{u_a}&=p_n-\bar p_a\quad(a\geq2),\\
 e_{y_k}&=x_{k+1}-x_k,&
 e_{z_k}&=p_k-p_{k+1}+x_k-x_{k+1}.
\end{aligned}
\end{equation}
These are the framing, arrow, and simple-root directions of
\cite[Definition~5.1]{SS19}.  Direct pairing gives
\begin{equation}
\label{eq:gln-duality}
 (e_i,\alpha_j)=\delta_{ij}\quad(j\text{ mutable}),\qquad
 (e_i,\alpha_j)=0\quad(j\text{ frozen}).
\end{equation}
It follows that $b_{ij}=(e_i,e_j)$ on the principal part and that
\eqref{eq:gln-compatibility} holds.  The map
$\lambda\mapsto\sum_i(\lambda,\alpha_i)e_i$, with $i$ mutable,
retracts the lattice onto the span of the $e_i$, proving primitivity.
\end{proof}

Let $\mathcal T_+$ have based multiplication
$Z_\lambda Z_\mu=v^{(\lambda,\mu)}Z_{\lambda+\mu}$, and let
$\mathcal A_{n,r}\subseteq\mathcal U_{n,r}$ be the quantum cluster
and upper quantum cluster algebras of this pair, with frozen
variables inverted.  For
$\beta_i^\pm=\sum_j[\pm b_{ji}]_+\alpha_j$, the initial exchange
has shifts
\begin{equation}
\label{eq:gln-exchange-shifts}
 \begin{gathered}
 Z_{\alpha_i}Z_{\alpha_i}'
 =v^{s_i^+}Z_{\beta_i^+}+v^{s_i^-}Z_{\beta_i^-},\qquad
 s_i^\pm=(\alpha_i,\beta_i^\pm),\\
 \begin{array}{c|rr}
 i&s_i^+&s_i^-\\\hline
 u_1&-n&1-n\\
 u_a\ (a\geq2)&0&1\\
 y_k&2k-1&2k\\
 z_k&-2k&1-2k
 \end{array}
 \end{gathered}
\end{equation}
In particular, $s_i^+-s_i^-=-1$ in every direction.

\subsection{The Whittaker isomorphism}

In the full lattice, the directions of \cite[Definition~5.1]{SS19}
are
\begin{equation}
\label{eq:SS-full-directions}
\begin{aligned}
 e_{u_1}^{\mathrm{full}}&=p_n-\zeta_f,&
 e_{u_a}^{\mathrm{full}}&=p_n-\bar p_a\quad(a\geq2),\\
 e_{y_k}^{\mathrm{full}}&=x_{k+1}-x_k,&
 e_{z_k}^{\mathrm{full}}&=p_k-p_{k+1}+x_k-x_{k+1}.
\end{aligned}
\end{equation}
The first row comes from the framing and gauge arrows; the second
comes from the $n-1$ simple roots of the central gauge node.
Write $\mathbf e_{n,r}^{\mathrm{full}}$ for this ordered list and
$\mathbf e_{n,r}$ for its specialization at $\zeta_f=0$.
The integral change $p'_i=p_i-\zeta_f$,
$\bar p'_a=\bar p_a-\zeta_f$ identifies the full list with
$\mathbf e_{n,r}$ in a direct summand.  Thus both are primitive by
Theorem~\ref{thm:gln-compatible-pair}.

Let $\mathbb T_{\mathrm{SS}}^{\mathrm{full}}$ be the full torus and
$\mathbb T_{\mathrm{SS}}$ the torus obtained by setting $Y_{\zeta_f}=1$.  For each seed
use the inverse-image charts and the finite intersection
$\mathcal L^{\mathrm{SS}}_{\mathbf e}$ defined in
\eqref{eq:SS-chart-pullback}--\eqref{eq:SS-Laurent-bound}.
Set
\[
 \mathcal K_Q^{\mathrm{full}}
 :=K^{(G_{Q,\mathcal O}\rtimes\mathbb C^\times_{\mathrm{rot}})
                 \times F_Q}(\mathcal R_Q^{\mathrm{cl}}).
\]
The algebraic Whittaker transform is the product of the pure-gauge
Whittaker transforms for $GL_n$ and the $r-1$ outer $GL_1$ factors,
extended over the representation ring of the flavor torus.  On the Toda generators
it is characterized by the dressed minuscule monopole formulas
\cite[Eqs.~(4.20)--(4.25)]{SS19}.  We write
$\mathsf W_Q^{\mathrm{SS}}$ for this transform with
$q_{\mathrm{SS}}=v^{-1}$ and with the rotation pullback
\eqref{eq:SS-rotation-centre-comparison} included.

\begin{theorem}[Schrader--Shapiro, {\cite[Thm.~5.13]{SS19}}]
\label{thm:SS-whittaker}
For every $n\geq2$ and $r\geq1$, the algebraic Whittaker transform
restricts to an isomorphism over the representation ring of the
loop-rotation and flavor tori:
\[
 (\mathsf W_Q^{\mathrm{SS}})^{\mathrm{full}}:
 \mathcal L^{\mathrm{SS}}_{\mathbf e_{n,r}^{\mathrm{full}}}
 \xrightarrow{\sim}\mathcal K_Q^{\mathrm{full}}.
\]
\end{theorem}

\begin{proof}
This is \cite[Thm.~5.13]{SS19} for the seed
\eqref{eq:SS-full-directions}. The quiver has no one-cycles,
its single flavor vertex has valency one, and there are no arrows
between flavor vertices, as required in \cite[Sec.~2.1]{SS19}.
With the inverse-image chart convention of
\eqref{eq:SS-chart-pullback}, the Laurentness conditions agree with
those used in the proofs of \cite[Lem.~4.2 and
Thm.~5.13]{SS19}. The coefficient ring is integral:
the source tori are defined over $\mathbb Z[q_{\mathrm{SS}}^{\pm1}]$
in \cite[Eq.~(3.1)]{SS19}, and the cited Whittaker isomorphism
is taken over that ring and the representation ring of the flavor torus.
The formal dilogarithms define rational automorphisms; their
binomial factors are not inverted in the resulting Laurent
intersection or Coulomb-branch algebra.
\end{proof}

\subsection{Normalization and signs}
\label{sec:gln-signs-upper}

For the shifts in \eqref{eq:integral-component-transport}, write
\[
 c([\lambda])=(n-1)\operatorname{val}_t\det A.
\]
The resulting Koszul-perverse heart $\mathcal C_{n,r}$ has integral
cohomological and loop degrees, with central cocharacter $\eta(c)=(cI_n,1,\ldots,1)$.
Put $\kappa_k=\binom{k}{2}$.
If $Q_k=L_0/L$ is the tautological quotient on $\Gr_{\omega_k}$
and $i_{\omega_k}$ is the immersion of
$\mathcal R_{\omega_k}^{\mathrm{cl}}$ into the derived space of triples, then
\begin{equation}
\label{eq:gln-positive-simple}
 \mathcal P_{\omega_k,\ell}=
 (i_{\omega_k})_*
 \bigl(\mathcal O_{\mathcal R_{\omega_k}^{\mathrm{cl}}}
          \otimes(\det Q_k)^\ell\bigr)
       [-\kappa_k-k\ell]\{\kappa_k\}.
\end{equation}
Indeed, the formal shifts in \cite[equation~(8.3)]{CW23} are
$[k(n-k)/2-k\ell]\{-k(n-k)/2\}$, and
\eqref{eq:integral-component-transport} subtracts $k(n-1)/2$ from the
cohomological degree.  Set
\[
 \mathsf U_a=\mathcal P_{\gamma_a,0},\qquad
 \mathsf Y_k=\mathcal P_{\omega_k,1},\qquad
 \mathsf Z_k=\mathcal P_{\omega_k,0}.
\]
The objects $\mathsf U_a$, $\mathsf D$, and $\mathsf H_a$ are normalized by
\begin{align*}
 \mathsf U_1&=(i_{\gamma_1})_*
  \mathcal O_{\mathcal R_{\gamma_1}^{\mathrm{cl}}}
                       [-\kappa_n]\{\kappa_n\},\\
 \mathsf U_a&=(i_{\gamma_a})_*
  \mathcal O_{\mathcal R_{\gamma_a}^{\mathrm{cl}}}\quad(a\geq2),\\
 \mathsf D&=(i_0)_*
  (\mathcal O_{\mathcal R_0^{\mathrm{cl}}}\otimes\det A)[-n]\{n\},\\
 \mathsf H_a&=(i_0)_*
  (\mathcal O_{\mathcal R_0^{\mathrm{cl}}}\otimes h_a).
\end{align*}

Put $q=v^{-1}$.  The Schrader--Shapiro torus has multiplication
$Y_\lambda^{\mathrm{SS}}Y_\mu^{\mathrm{SS}}
=(-v)^{(\lambda,\mu)}Y_{\lambda+\mu}^{\mathrm{SS}}$.
For
$\lambda=\sum_i(a_ip_i+b_ix_i)+
\sum_{a=2}^r(c_a\bar p_a+d_a\bar x_a)$, define
\begin{equation}
\label{eq:gln-quadratic-refinement}
 R_{n,r}(\lambda)=
 \sum_i a_ib_i+\sum_{a=2}^r c_ad_a
 +\sum_i a_i+\sum_i(i-1)b_i\pmod2.
\end{equation}

\begin{proposition}
\label{prop:gln-odd-refinement}
The function $R_{n,r}$ is a quadratic refinement which is odd on
every mutable direction of every seed.  Its initial values are
\begin{equation}
\label{eq:gln-initial-parities}
 \begin{array}{c|cccccc}
 i&u_1&u_a\ (a\geq2)&y_k&z_k&d&h_a\\\hline
 R_{n,r}(\alpha_i)&\kappa_n&0&\kappa_k&\kappa_k&n&0
 \end{array}\pmod2.
\end{equation}
\end{proposition}

\begin{proof}
Polarization of the bilinear terms in
\eqref{eq:gln-quadratic-refinement} gives the skew form modulo two.
Evaluation on \eqref{eq:gln-A-basis} gives the table.  The values
on $p_n$ and $p_n-\bar p_a$ are one.  On $x_{k+1}-x_k$ and
$p_k-p_{k+1}+x_k-x_{k+1}$ the last linear term is odd, while the
remaining terms are even.  Thus $R_{n,r}(e_i)=1$ initially, and
Lemma~\ref{lem:odd-refinement-mutation} proves mutation invariance.
\end{proof}

Define
\begin{equation}
\label{eq:gln-rho}
 \rho(Z_\lambda)=(-1)^{R_{n,r}(\lambda)}Y_\lambda^{\mathrm{SS}}.
\end{equation}
The quadratic-refinement identity makes $\rho$ an algebra isomorphism
of quantum tori.  As in the loop-rotation comparison in the preliminaries,
we include the change of splitting
$(g,c)\mapsto(g\eta(c^{-1}),c)$ in the Whittaker transform.
Thus a character of $\eta$-weight $k$ acquires a factor $v^k$.

\begin{proposition}
\label{prop:gln-simple-frame}
The Whittaker transform composed with $\rho$ sends the initial
variables to
$[\mathsf U_a],[\mathsf Y_k],[\mathsf Z_k],[\mathsf D],[\mathsf H_a]$
in their specified order.
\end{proposition}

\begin{proof}
The initial mutable coweights have nonnegative pairing with every
weight of $N_{n,r}$, so the contribution of $N_{n,r}$ to each
localization coefficient is $1$. By \cite[equations~(4.24)--(4.25)]{SS19}, the
Whittaker images of $Y^{\mathrm{SS}}_{\alpha_{z_k}}$ and
$Y^{\mathrm{SS}}_{\alpha_{y_k}}$, after the change of splitting
$(g,c)\mapsto(g\eta(c^{-1}),c)$, are respectively
\[
 v^{\kappa_k}[\mathcal O],\qquad
 (-1)^kv^{\kappa_k}[\mathcal O\otimes\det Q_k].
\]
Formula~\eqref{eq:gln-positive-simple} requires the additional
sign $(-1)^{\kappa_k}$ in both cases.  The same calculation for
$u_1$ gives $(-1)^{\kappa_n}$; the variables $u_a$, $a\geq2$,
require no additional sign.  Finally the character images are $v^n\det A$
and $h_a$, requiring signs $(-1)^n$ and one.  These are exactly
\eqref{eq:gln-initial-parities}, with no residual powers of $v$.
\end{proof}

\subsection{The upper algebra}

Write $\mathcal L^{\mathrm{SS}}_{n,r}
=\mathcal L^{\mathrm{SS}}_{\mathbf e_{n,r}}$ and
$\mathcal K_{n,r}
=K^{G_{n,r,\mathcal O}\rtimes\mathbb C^\times_{\mathrm{rot}}}
 (\mathcal R_{G_{n,r},N_{n,r}}^{\mathrm{cl}})$.

\begin{theorem}
\label{thm:gln-upper-Coulomb}
There are isomorphisms over
$\mathbb Z[v^{\pm1}]$
\begin{equation}
\label{eq:gln-upper-Coulomb}
 \mathcal U_{n,r}\xrightarrow{\sim}\mathcal K_{n,r}
             \xleftarrow{\sim}K_0(\mathcal C_{n,r}),
\end{equation}
where $\mathcal K_{n,r}$ is the equivariant $K$-theory convolution
ring of the classical space of triples. The first isomorphism is induced
by the Whittaker transform composed with $\rho$. Only the $r$ frozen
character variables are inverted.
\end{theorem}

\begin{proof}
The oddness of $R_{n,r}$ and Lemma~\ref{lem:dilog-transport} identify the Laurent
charts of the two quantum tori through $\rho$.
Duality~\eqref{eq:gln-duality} supplies the hypothesis of
Proposition~\ref{prop:v-upper-A-comparison}, which gives
$\rho(\mathcal U_{n,r})=\mathcal L^{\mathrm{SS}}_{n,r}$ over
$\mathbb Z[v^{\pm1}]$.

The specialization of the framing character to one is split.
In the full framed lattice adjoin the radical generator $\zeta_f$
and replace $e_{u_1}$ by $p_n-\zeta_f$.  The integral change
$p'_i=p_i-\zeta_f$, $\bar p'_a=\bar p_a-\zeta_f$ expresses every
mutable direction without $\zeta_f$.  The full torus and every
Laurent chart therefore split off
$\mathbb Z[Y_{\zeta_f}^{\pm1}]$, as does their intersection
coefficient by coefficient.  On the geometric side the coordinates
\[
 (A,h_2,\ldots,h_r,f)\longmapsto
 (Af^{-1},h_2f^{-1},\ldots,h_rf^{-1},f)
\]
split off the trivially acting framing factor, with the same
Laurent character ring.  Theorem~\ref{thm:SS-whittaker}, followed by specialization at
$Y_{\zeta_f}=1$, gives the Whittaker isomorphism for $G_{n,r}$
\begin{equation}
\label{eq:SS-effective-whittaker}
 \mathsf W_Q^{\mathrm{SS}}:
 \mathcal L^{\mathrm{SS}}_{n,r}\xrightarrow{\sim}\mathcal K_{n,r}.
\end{equation}
The comparisons \eqref{eq:heart-K0} and
\eqref{eq:derived-classical-K-comparison}, with the weight-one
cocharacter $\eta$ and the shifts prescribed by $c$, give the
isomorphism $K_0(\mathcal C_{n,r})\to\mathcal K_{n,r}$.  Proposition~\ref{prop:gln-simple-frame}
identifies the initial generators.
\end{proof}

The sign on $Z_{\alpha_d}$ in \eqref{eq:gln-rho} is $(-1)^n$.

\begin{remark}
\label{rem:acyclic-small-rank}
For $n=2$, the mutable vertices are $u_a$, $y_1$, and $z_1$,
with arrows $u_a\to y_1\overset{2}{\longrightarrow}z_1\to u_a$.
The commuting mutations at $u_1,\ldots,u_r$ reverse all arrows incident to those vertices and
replace the signed multiplicity of $y_1\to z_1$ by $2-r$.
For $r=1,2$ the result is acyclic; for $r=3$ one further mutation at
$y_1$ gives the acyclic quiver $u_a\to y_1\to z_1$.
Thus \cite[Thm.~7.5]{BZ05} gives an independent proof of
(\ref{eq:unified-A-equals-U}) in these cases.

For $r=4$, mutations at $u_1,u_2,y_1$ give the tetrahedral
triangulation quiver of the four-punctured sphere, with arcs
$(u_1,u_2,u_3,u_4,y_1,z_1)=(AB,CD,AD,BC,AC,BD)$.
This is type $D_4^{(1,1)}$ \cite[Rem.~12.8]{FST08}.
Its principal matrix has rank two, preserved by mutation.
An acyclic representative would orient a six-vertex Dynkin or
extended Dynkin diagram \cite[Thm.~12.1]{FST08}; each such diagram
contains an induced four-vertex path, whose skew adjacency matrix
has determinant one, and hence has matrix rank at least four.
Thus no acyclic seed exists.
For $r\geq5$, mutation at $u_1,\ldots,u_r$ produces $r-2\geq3$ parallel
arrows, so the principal part is mutation-infinite
\cite[Thm.~2.6]{FSTFinite}.
\end{remark}

\section{The initial monoidal seed and its mutations}
\label{sec:initial-monoidal-seed}
\label{sec:category}

Let $\mathcal C_{n,r}$ be the Koszul-perverse heart of the space of
triples for $(G_{n,r},N_{n,r})$, with loop rotation pulled back to the
double cover and with the cohomological and loop shifts of \eqref{eq:integral-component-transport}. We
use the convolution structure of \cite{CW23}; in particular, all
pullbacks on the space of triples are taken in its stable
infinity-category of coherent sheaves. Under the hypotheses of Lemma~\ref{lem:cat-positive-kernel},
the convolution kernels are pushforwards from classical correspondences.

Write $V=\mathbb C^n$ and $L_0=V[[t]]$. We retain the
objects, shifts, and notation of \eqref{eq:gln-positive-simple}. On any closed minuscule orbit,
write $i_\lambda$ for the immersion of
$\mathcal R_\lambda^{\mathrm{cl}}$ into the space of triples. The simple associated with the trivial Levi representation is
\[
 \mathcal P_{\lambda,0}
 =i_{\lambda*}\mathcal O_{\mathcal R_\lambda^{\mathrm{cl}}}
 \bigl[(\dim\operatorname{Gr}_\lambda-c([\lambda]))/2\bigr]
 \bigl\{(c([\lambda])-\dim\operatorname{Gr}_\lambda)/2\bigr\}.
\]
This includes $\lambda=\omega_{n-1}-\gamma_a$.

For pairwise $v$-commuting real
simples, $\odot$ denotes the shift of their ordered convolution whose
class is the based quantum monomial. We also use this notation when an
invertible object associated with a one-dimensional $G_{n,r}$-representation,
with its prescribed shifts, is convolved with a simple object: the two commute
up to a unique loop shift, which defines the same normalization. We set
\begin{equation}
\label{eq:cat-boundary}
 \mathsf Z_0=\mathsf Y_0=\mathbf1,\qquad
 \mathsf Z_n=\bigodot_{a=1}^r\mathsf U_a,\qquad
 \mathsf Y_n=\mathsf D\odot\mathsf Z_n,\qquad
 \mathsf H_1=\mathbf1.
\end{equation}
Theorems~\ref{thm:CW-heart} and~\ref{thm:CW-monoidal-structure}
give a $\mathbb C$-linear finite-length rigid monoidal heart with
biexact convolution and homogeneous renormalized $r$-matrices.
The loop shift is an exact central monoidal action. Rigidity gives
separable triple products and nonzero $r$-matrices for nonzero
objects, by \cite[Cor.~4.7 and Prop.~4.9]{CW19}. These structures
are preserved by the monoidal equivalence
\eqref{eq:integral-component-transport} and satisfy the graded
axioms of Definition~\ref{def:homogeneous-r-matrices}.
Section~\ref{sec:convolution-correspondence} gives local finiteness and
$\operatorname{End}(S)=\mathbb C$ for every simple object $S$,
including the monoidal unit.

The objects $\mathsf D$ and $\mathsf H_a$ are invertible:
convolution adds characters, so the opposite characters with opposite
normalizing shifts give their inverses. Convolving on either side with
one of these objects preserves simplicity. Its nonzero homogeneous
$r$-matrix with any simple object is therefore an isomorphism.
Thus $\mathsf D$ and $\mathsf H_a$ are real and $v$-commute with every
simple object.

By Theorem~\ref{thm:gln-upper-Coulomb}, the Grothendieck ring is
$\mathcal U_{n,r}$ inside the skew field of fractions of the
initial quantum torus. In particular it is a domain, the coefficient map
$\mathbb Z[v^{\pm1}]\to K_0(\mathcal C_{n,r})$ is injective, and distinct initial based monomials have
distinct classes. Proposition~\ref{prop:gln-simple-frame} identifies
these classes with the specified initial objects.

\subsection{Convolution kernels on classical supports}

\begin{lemma}
\label{lem:cat-positive-kernel}
\label{lem:passive-block-derived-base-change}
\label{lem:positive-lattice-reality}
Suppose that the classical support of the first factor satisfies
$\mathcal R_{\lambda}^{\mathrm{cl}}=\mathcal T_{\lambda}$, where
$\mathcal T$ is the tautological bundle on the affine Grassmannian.
The pullback of the two classical supports to the derived convolution
correspondence is classical and is naturally equivalent to
\[
 \operatorname{Gr}_{\lambda}\widetilde\times
             \mathcal R_{\mu}^{\mathrm{cl}}.
\]
More precisely, write $d$ for the pullback map in the derived
convolution correspondence, $\widetilde m$ for its product map,
and $\cl$ for the product of the two classical support immersions.
Let $\cl'$ and $\widehat d$ be their derived base changes. For a
vector bundle $\mathcal L$ on the product of the classical supports,
\begin{equation}
\label{eq:passive-block-kernel-and-pushforward}
\begin{aligned}
 Ld^*\cl_*\mathcal L&\simeq\cl'_*\widehat d^*\mathcal L,\\
 R\widetilde m_*Ld^*\cl_*\mathcal L
   &\simeq R(\widetilde m\circ\cl')_*\widehat d^*\mathcal L.
\end{aligned}
\end{equation}
Thus the kernel is the pushforward of a bundle from the classical
correspondence into the derived correspondence.

If $\lambda$ and $\mu$ are nonnegative integral combinations of
$\omega_1,\ldots,\omega_n$ and $\gamma_1,\ldots,\gamma_r$, the conditions on
the vector in $N_{n,r}((t))$ depend only on the endpoint lattice. Thus the calculation
is the flat pullback of the corresponding affine-Grassmannian
convolution calculation along the endpoint tautological bundle.
In particular every $\mathcal P_{\omega_k,\ell}$,
$1\leq k<n$, $\ell\in\mathbb Z$, and every $\mathsf U_a$ is real.
Moreover $\bigodot_a\mathsf U_a=\mathcal P_{\omega_n,0}$ and
$\mathsf D\odot\mathcal P_{\omega_n,0}=\mathcal P_{\omega_n,1}$.
\end{lemma}

\begin{proof}
Pulling back the two support immersions in the Cartesian squares
of \cite[Eq.~(5.10), Prop.~5.9]{CW23} identifies their pullback
to the derived convolution correspondence with
\[
 \mathcal R_{\lambda}^{\mathrm{cl}}
 \mathbin{\times^{\mathbf R}_{\mathcal T_{\lambda}}}
 \bigl(\operatorname{Gr}_{\lambda}\widetilde\times
                    \mathcal R_{\mu}^{\mathrm{cl}}\bigr).
\]
The first map to $\mathcal T_{\lambda}$ is the identity. The fiber
product therefore equals the second factor and is classical.
The same proper-base-change argument applies to vector bundles
and computes their convolution by pushforward from this
classical correspondence into the derived target.

For the coweights in the second assertion and each $a$, the successive
lattices $L_a\Lambda$ in the $a$th summand of $N_{n,r}((t))$
form a decreasing chain. A vector belonging to the final lattice
belongs to every intermediate lattice. The correspondence with these vectors
is therefore the pullback of the Grassmannian correspondence along
the bundle of vectors in the endpoint lattices. On every
finite lattice quotient this map is a vector bundle and hence flat.
Proper flat base change proves the assertion for pushforwards and
the projection formula handles the bundle $\mathcal L$. The
identifications commute with the flat changes of lattice quotient
and the proper changes of Schubert bound in \cite[Sec.~4.1]{CW23},
so they give the asserted equivalences on the space of triples.

For $r=1$, the reality of every
$\mathcal P_{\omega_k,\ell}$ is \cite[Prop.~8.9]{CW23}.
The preceding flat-base-change calculation identifies its
convolution square for $N_{n,r}$ with the flat pullback of
the kernel and its pushforward for $r=1$, with the vectors in the
remaining summands varying in their endpoint bundles. The description of simples in
\cite[Prop.~3.27 and Thm.~4.1]{CW23} identifies the resulting
normalized object as the simple with the same Schubert and
Levi labels. Thus reality is preserved. For the central point
modifications $\gamma_a$, the convolution morphism is an
isomorphism; the same description proves reality of $\mathsf U_a$
and identifies $\bigodot_a\mathsf U_a$ with
$\mathcal P_{\omega_n,0}$. Tensoring by the determinant line,
with the normalization in \eqref{eq:gln-positive-simple}, gives
the identity for $\mathsf Y_n$.

\end{proof}

\begin{proposition}
\label{prop:cat-initial-seed}
\label{lem:initial-real-commuting}
The objects $\mathsf U_a$, $\mathsf Y_k$, $\mathsf Z_k$,
$\mathsf D$, and $\mathsf H_a$ form a quantum monoidal seed for
$(\Lambda_{n,r},\widetilde B_{n,r})$. Its frozen objects are
invertible. Every normalized monomial in the seed objects, allowing
arbitrary integral frozen powers, is simple, and distinct exponent
vectors give nonisomorphic objects even up to loop shift.
\end{proposition}

\begin{proof}
The classification of simples in \cite[Thm.~4.1]{CW23} gives
simplicity of all the displayed objects.
Lemma~\ref{lem:cat-positive-kernel} and the preceding argument for
$\mathsf D$ and $\mathsf H_a$ show that every seed object is real and
that $\mathsf D$ and $\mathsf H_a$ are invertible.

By Proposition~\ref{prop:gln-simple-frame}, the classes of the
displayed objects are the initial toric-frame generators
with exponents $\alpha_i$. They therefore
$v$-commute with form $\Lambda_{n,r}$, and the integral-basis
property of the $\alpha_i$ gives independence of based monomials.
The $r$-matrix criterion \cite[Cor.~1.4]{CW23}, applied with one
real simple factor, lifts these commutation relations to
isomorphisms of the convolutions and makes their products simple.
Its iterated form \cite[Cor.~4.15]{CW19} proves simplicity of
every seed monomial. The same argument applies to $\mathsf D^{-1}$
and $\mathsf H_a^{-1}$, so arbitrary frozen powers are allowed. Independence
in the quantum torus proves the final assertion.
\end{proof}

\subsection{The initial exchange sequences}

\begin{theorem}
\label{thm:cat-initial-exchanges}
The initial mutation candidates are
\begin{equation}
\label{eq:hf-mutation-candidates}
 \mathsf U'_a=\mathsf D\odot\mathcal P_{\omega_{n-1}-\gamma_a,0},
 \qquad\mathsf Y'_k=\mathcal P_{\omega_k,-1},
 \qquad\mathsf Z'_k=\mathcal P_{\omega_k,2}.
\end{equation}
For $n\geq2$ and $r\geq1$, the following are short exact sequences
in $\mathcal C_{n,r}$:
\begin{align}
0&\longrightarrow(\mathsf D\odot\mathsf Z_{n-1})\{-1\}
 \longrightarrow(\mathsf U_1*\mathsf U'_1)\{n-1\}
 \longrightarrow\mathsf Y_{n-1}\longrightarrow0,
 \label{eq:cat-outer-first}\\
0&\longrightarrow(\mathsf D\odot\mathsf Z_{n-1})\{-1\}
 \longrightarrow(\mathsf U_a*\mathsf U'_a)\{-1\}
 \longrightarrow\mathsf H_a\odot\mathsf Y_{n-1}
 \longrightarrow0 \quad(2\leq a\leq r),
 \label{eq:cat-outer-other}\\
0&\longrightarrow(\mathsf Z_{k-1}\odot\mathsf Z_{k+1})\{-1\}
 \longrightarrow(\mathsf Y_k*\mathcal P_{\omega_k,-1})\{-2k\}
 \longrightarrow\mathsf Z_k\odot\mathsf Z_k\longrightarrow0,
 \label{eq:cat-q-first}\\
0&\longrightarrow(\mathsf Y_k\odot\mathsf Y_k)\{-1\}
 \longrightarrow(\mathsf Z_k*\mathcal P_{\omega_k,2})\{2k-1\}
 \longrightarrow\mathsf Y_{k-1}\odot\mathsf Y_{k+1}
 \longrightarrow0 \quad(1\leq k<n).
 \label{eq:cat-q-second}
\end{align}
They realize the initial quantum mutations at $u_a$, $y_k$, and
$z_k$, respectively.
\end{theorem}

\begin{proof}
The exchange sequence for $(GL_n,V)$ is \cite[Prop.~8.5]{CW23}.
We apply its Cartier-divisor construction to the $a$th summand of $N_{n,r}$,
retaining $G_{n,r}$-equivariance and the loop grading.
Over
$\operatorname{Gr}_{\omega_{n-1}}$ put $S=L/tL_0$ and
$Q=L_0/L$. The bundle $S$ has rank one and $Q$ has rank $n-1$.
The convolution of the point modification $\gamma_a$ with
$\omega_{n-1}-\gamma_a$ has classical support
\[
 D_a=\{(L,s_1,\ldots,s_r):s_b\in L\text{ for every }b,
                                      \ s_a\in tL_0\}.
\]
Indeed, the first modification scales only the $a$th summand of $N_{n,r}$, and the
second restores its endpoint lattice to $L$. The endpoint uniquely
determines the intermediate point modification. Lemma
\ref{lem:cat-positive-kernel} identifies the convolution kernel
with the pushforward of $\mathcal O_{D_a}$.

The image of $s_a$ in $S$ is a regular section of $S\otimes h_a^{-1}$,
with $h_1=1$. It is a fiberwise linear coordinate, so its zero locus
is the Cartier divisor $D_a$. On $\mathcal R_{\omega_{n-1}}^{\mathrm{cl}}$
we have the distinguished triangle associated to this divisor,
\begin{equation}
\label{eq:cartier-v}
 \mathcal O\otimes S^\vee\otimes h_a
 \longrightarrow\mathcal O\longrightarrow\mathcal O_{D_a}
 \longrightarrow(\mathcal O\otimes S^\vee\otimes h_a)[1],
\end{equation}
and
$\det V\otimes S^\vee\simeq\det Q$.
Rotating this triangle and tensoring by $\det V$ therefore gives
a triangle with terms
$\det V$, $\det V\otimes\mathcal O_{D_a}$, and
$h_a\det Q[1]$.

For $a=1$, the loop shifts before replacing convolution with $\mathsf D$
by normalized products are those of
\cite[Eq.~(8.6)]{CW23}: $-n$ on
$\mathsf D*\mathsf Z_{n-1}$ and $n$ on
$\mathsf U_1*(\mathsf D*\mathcal P_{\omega_{n-1}-\gamma_1,0})$.
The bundles of the vectors $s_b$, $b\ne1$, are pulled back from the
endpoint and do not change these loop shifts. The pairings of $\det A$ with
$\omega_{n-1}$ and $\omega_{n-1}-\gamma_1$ are $n-1$ and $-1$,
respectively. Replacing these two ordered products by normalized
products gives $-1$ and $n-1$, proving
\eqref{eq:cat-outer-first}.

For $a\geq2$, the modification $\gamma_a$ has trivial $GL_n$
determinant and no normalizing shift. The object
$\mathcal P_{\omega_{n-1}-\gamma_a,0}$ has shifts
$[-\kappa_{n-1}]\{\kappa_{n-1}\}$.
Multiplication by $\mathsf D$, normalization using
$\langle\det A,\omega_{n-1}-\gamma_a\rangle=n-1$, and the displayed
middle shift $-1$ give shifts
$[-\kappa_{n-1}-n]\{\kappa_{n-1}\}$ on
$\det V\otimes\mathcal O_{D_a}$. These are the shifts of the
subobject in \eqref{eq:cat-outer-other}. The quotient after rotating
\eqref{eq:cartier-v} has cohomological degree one larger, namely
$-\kappa_{n-1}-(n-1)$, which is precisely that of
$\mathsf H_a\odot\mathsf Y_{n-1}$. This proves
\eqref{eq:cat-outer-other}. The factor $h_a$ comes from the
equivariance of the section. Equivalently, conversion from the
linearization given by $Ah_a^{-1}$ to that given by $A$ uses
\begin{equation}
\label{eq:block-character-conversion}
 \det(Ah_a^{-1})=\det(A)h_a^{-n},\qquad
 \det Q_{Ah_a^{-1}}=(\det Q_A)h_a^{-(n-1)}.
\end{equation}

For the mutations at $y_k$ and $z_k$, use \cite[Prop.~8.8]{CW23}.
Its two sequences hold for every $1\leq k<n$ and every determinant
dressing $\ell\in\mathbb Z$. Their convolution kernels depend
on the vectors $s_a$ only through the endpoint lattices, so
Lemma~\ref{lem:cat-positive-kernel} gives the same sequences for
$(G_{n,r},N_{n,r})$.
The equivalence \eqref{eq:integral-component-transport} is monoidal
and hence preserves the
sequences. Taking $\ell=0$ in the first gives
\eqref{eq:cat-q-first}; taking $\ell=1$ in the second and shifting
by $\{-1\}$ gives \eqref{eq:cat-q-second}. For fixed determinant dressing
$\ell$, the factors with indices $k-1$ and $k+1$ have zero pairing,
so replacing their ordered convolutions by normalized products adds no shift.
The sequences of \cite[Prop.~8.8]{CW23} and their pullbacks are triangles whose terms lie in
the heart, by \cite[Thms.~4.1 and 5.1]{CW23}; they are therefore
short exact sequences in $\mathcal C_{n,r}$.

Finally, the shifts in \eqref{eq:gln-exchange-shifts} have
$s_i^+-s_i^-=-1$ and agree with the displayed sequences. They
are exactly the shifts of
Definition~\ref{def:quantum-monoidal-seed},
\eqref{eq:abstract-monoidal-mutation}, for compatibility
diagonal $2I$. The candidates are simple, and compatibility
excludes zero exchange columns. Theorem~\ref{thm:gln-upper-Coulomb} embeds
$K_0(\mathcal C_{n,r})$ into this skew field, which contains
the initial toric frame and its mutations. Thus
Lemma~\ref{lem:local-mutation-test} applies: the displayed sequences are
quantum monoidal mutations, their new objects are real simple,
and the reverse-order exchange sequences also hold.
\end{proof}

\subsection{Arbitrary sequences of monoidal mutations}

\begin{theorem}
\label{thm:cat-all-mutations}
For every $n\geq2$ and $r\geq1$, every finite mutation
sequence from Proposition~\ref{prop:cat-initial-seed} is realized
by quantum monoidal seeds and short exact sequences in
$\mathcal C_{n,r}$. Every cluster variable is represented by a
real simple object, and every quantum cluster monomial, with
arbitrary integral frozen powers, is represented by a simple
normalized convolution.
\end{theorem}

\begin{proof}
We apply Lemma~\ref{lem:frozen-localized-factorization}. The categorical hypotheses were
verified above, and Proposition~\ref{prop:cat-initial-seed}
gives the injective initial toric frame and invertible frozen
objects. Theorem~\ref{thm:cat-initial-exchanges} supplies every
initial one-step mutation. Theorem~\ref{thm:gln-compatible-pair}
gives compatibility with diagonal $2I$, which
excludes zero mutable columns after every matrix mutation, and
the commutation forms remain even under integral matrix mutation.

Theorem~\ref{thm:gln-upper-Coulomb} gives in this same initial toric frame
\[
 \mathcal U_{n,r}\simeq K_0(\mathcal C_{n,r}).
\]
The embedding into this skew field shows that
$K_0(\mathcal C_{n,r})$ is a domain and that the coefficient map
$\mathbb Z[v^{\pm1}]\to K_0(\mathcal C_{n,r})$ is injective. The quantum Laurent
phenomenon gives $\mathcal A_{n,r}\subseteq\mathcal U_{n,r}$
and hence the injection
\[
 \mathbb Q(v)\otimes_{\mathbb Z[v^{\pm1}]}\mathcal A_{n,r}
 \lhook\joinrel\longrightarrow
 \mathbb Q(v)\otimes_{\mathbb Z[v^{\pm1}]}K_0(\mathcal C_{n,r}).
\]
Lemma~\ref{lem:frozen-localized-factorization} now gives every
finite quantum monoidal mutation sequence and the asserted simple
representatives of variables and localized cluster monomials.
\end{proof}

\section{Integral generation}
\label{sec:integral-generation}

Let $G_{n,r}=(GL_n\times(\mathbb C^\times)^r)/
\{(z^{-1}I_n,z,\ldots,z)\}$ and $N_{n,r}=(\mathbb C^n)^{\oplus r}$,
with the seed and the integral Coulomb-branch realization constructed
above. Write $\mathcal K_{n,r}$ for the latter algebra, identified with
$K_0(\mathcal C_{n,r})$. We prove the following equality inside that
realization.

\begin{theorem}\label{thm:integral-generation}\label{thm:gen-higher-rank-main}
Let $n\geq2$ and $r\geq1$. If $n=2$ or $r\leq5$, then
\begin{equation}\label{eq:unified-A-equals-U}
 \mathcal A_{n,r}=\mathcal K_{n,r}=\mathcal U_{n,r}
\end{equation}
over $\mathbb Z[v^{\pm1}]$, with only the original frozen characters
inverted.
\end{theorem}

\subsection{Determinant twists}
\label{sec:gln-determinant-twists}

For a lattice $L$ commensurable with $L_0$, define the relative
determinant line, with the fixed $GL_n$-linearization, by
\begin{equation}\label{eq:relative-determinant-line}
 \mathscr L_L=
 \det\bigl(L_0/(L\cap L_0)\bigr)
 \otimes\det\bigl(L/(L\cap L_0)\bigr)^{-1}.
\end{equation}
We use its integer powers to twist the objects below.
We apply this construction to the lattice $A L_0=L_1\Lambda$ associated
with the projection $G_{n,r}\to GL_n$, $(A,h_2,\ldots,h_r)\mapsto A$.
Transitivity of relative determinants equips this line with a
convolution multiplication; for lattices $L\subseteq L_0$, this is
the compatibility of determinant lines with convolution used in \cite[proof of Proposition~2.10]{CW19}.
On the component indexed by $\lambda$,
tensoring with $\mathscr L$ and shifting by
$[-\operatorname{val}_t\det A]$ defines a monoidal autoequivalence
$\mathsf T$ of $\mathcal C_{n,r}$.  Indeed,
$\mathscr L$ has central weight $\operatorname{val}_t\det A$;
the indicated shift is precisely the shift required by the
Koszul-perverse normalization.  Its inverse tensors with
$\mathscr L^{-1}$ and uses the opposite shift.  On the initial
objects,
\begin{equation}
\label{eq:gln-determinant-action}
\begin{aligned}
 \mathsf T(\mathcal P_{\omega_k,\ell})
     &=\mathcal P_{\omega_k,\ell+1},&
 \mathsf T(\mathsf U_1)&=\mathsf D\odot\mathsf U_1,\\
 \mathsf T(\mathsf U_a)&=\mathsf U_a\quad(a\geq2),&
 \mathsf T(\mathsf D)&=\mathsf D,\qquad
 \mathsf T(\mathsf H_a)=\mathsf H_a.
\end{aligned}
\end{equation}
For the second equality, the normalized product subtracts the
loop shift $n$ from the product with
$\mathsf D=(\det A)[-n]\{n\}$, since
$(\alpha_d,\alpha_{u_1})=n$.

All monomials in the following lemma are based monomials.

\begin{lemma}
\label{lem:gln-frozen-rescaling}
Let two quantum seeds have the same frozen variables and let their
extended exchange matrices and commutation matrices be
$(B,L)$ and $(\bar B,\bar L)$.  Suppose
\[
 C=\begin{pmatrix}I&0\\F&I\end{pmatrix},\qquad
 \bar B=CB,\qquad L=C^{\mathsf T}\bar L C,
\]
where the first block corresponds to mutable indices, and suppose
their toric frames satisfy $M(a)=\bar M(Ca)$.
Then, after any common mutation sequence, the corresponding mutable
variables differ by invertible frozen monomials and powers of the
quantum parameter.  The two seeds generate the same quantum cluster
algebra when frozen variables are inverted.
\end{lemma}

\begin{proof}
This follows from the quantum quasi-homomorphism criterion and
its compatibility with mutation \cite[Propositions~2.7--2.8]{CHL24},
with all frozen variables inverted.
The same argument works over $\mathbb Z[v^{\pm1}]$: if
$c_k^\pm=-e_k+[\pm b_k]_+$ and
$\bar c_k^\pm=-e_k+[\pm\bar b_k]_+$, then
\[
 \zeta=Cc_k^+-\bar c_k^+=Cc_k^--\bar c_k^-
\]
is frozen, and compatibility gives $\bar L(\zeta,\bar b_k)=0$.
Thus the two exchange summands have the same normalization factor,
and the mutated frames satisfy $M'_k=\bar M'(e_k+\zeta)$.
The cited mutation argument iterates this relation, using only
integral lattice maps, frozen monomials, and powers of the quantum
parameter. Its proof does not require the indecomposability convention
in that source. Applying it to both $C$ and $C^{-1}$ gives the equality
over $\mathbb Z[v^{\pm1}]$ after frozen inversion.
\end{proof}

\begin{proposition}
\label{prop:gln-translation-stability}
The automorphism induced by $\mathsf T$ preserves
$\mathcal A_{n,r}$.  Consequently
\[
 [\mathcal P_{\omega_k,\ell}]\in\mathcal A_{n,r}
 \qquad(1\leq k<n,\ \ell\in\mathbb Z).
\]
\end{proposition}

\begin{proof}
The vertices $z_1,\ldots,z_{n-1}$ are pairwise nonadjacent.
The initial Q-system exchange sequences identify their simultaneous
mutation with the replacement
$\mathcal P_{\omega_k,0}\mapsto\mathcal P_{\omega_k,2}$.
After interchanging the labels $y_k$ and $z_k$, the resulting seed
has mutable objects
\[
 \mathsf U_a,\quad
 \mathcal P_{\omega_k,2},\quad
 \mathcal P_{\omega_k,1}.
\]
By \eqref{eq:gln-determinant-action}, applying $\mathsf T$ to the
initial seed gives this collection with $\mathsf U_1$ replaced by
$\mathsf D\odot\mathsf U_1$.

Put
$l=n-1$, let $K$ be the Cartan matrix of type $A_l$, and set
$J=\mathbf1_r e_l^{\mathsf T}$ and
$H=(0\ I_{r-1})$.  In the order $(u,y,z,d,h)$,
\begin{equation}
\label{eq:gln-block-exchange-matrix}
 B=\begin{pmatrix}
 0&J&-J\\
 -J^{\mathsf T}&0&K\\
 J^{\mathsf T}&-K&0\\
 \mathbf1_r^{\mathsf T}&0&-e_l^{\mathsf T}\\
 -H&0&0
 \end{pmatrix}.
\end{equation}
Mutating all $z$-columns and interchanging the $y,z$ labels restores
all mutable rows and all $h$-rows.  The $d$-row changes from
$(\mathbf1_r^{\mathsf T},0,-e_l^{\mathsf T})$ to
$(\mathbf1_r^{\mathsf T},e_l^{\mathsf T},-2e_l^{\mathsf T})$.
For example, before relabeling, its only new entries are
$-2e_l^{\mathsf T}$ in the $y$-block and
$e_l^{\mathsf T}$ in the $z$-block; the mutable-row calculation
uses the tridiagonal entries $2,-1$ of $K$.
Thus the resulting extended matrix is
\begin{equation}
\label{eq:gln-translation-matrix}
 \bar B=CB,\qquad C=I+E_{d,u_1}.
\end{equation}
The toric frame obtained by applying $\mathsf T$ has the same
commutation matrix as the initial frame, since $\mathsf T$ is a
graded monoidal autoequivalence.  Its relation to the mutated frame
is multiplication of the $u_1$ variable by the based frozen
monomial $d$, hence its matrix is $C^{\mathsf T}\bar\Lambda C$.
Lemma~\ref{lem:gln-frozen-rescaling} shows that the image of the initial
seed under $\mathsf T$ generates the same algebra as the mutated seed.
Therefore $\mathsf T(\mathcal A_{n,r})=\mathcal A_{n,r}$.
Applying all positive and negative powers of $\mathsf T$ to the
initial $z_k$ variables proves the assertion.
\end{proof}

\subsection{Determinant dressings and mutation sequences}
\label{subsec:gln-mixed-baxter}

Put $q=v^{-1}$, $h_1=1$, and
$d_{\mathrm{raw}}=\prod_{i=1}^n w_i=(-1)^n[\mathsf D]$.
Write $L_a$ for the shift of $h_a$ for $2\leq a\leq r$. Thus
\[
 D_iw_j=q^{2\delta_{ij}}w_jD_i,\qquad
 L_ah_b=q^{2\delta_{ab}}h_bL_a,
\]
and all other character--shift commutators are trivial. Set
\[
 U_1=D_1\cdots D_n\prod_{a=2}^rL_a,\qquad
 U_a=L_a^{-1}\quad(a\geq2),\qquad
 U_S=\prod_{a\in S}U_a.
\]
For a subset $I\subseteq[1,n]$, write $W_I=(w_i)_{i\in I}$ and put
$w_I=\prod_{i\in I}w_i$,
$D_I=\prod_{i\in I}D_i$, and
$\delta_I=\prod_{i\in I,j\notin I}(1-w_j/w_i)$.
For $0\leq k\leq n$ and $\ell\in\mathbb Z$, let
\begin{equation}
\label{eq:gln-mixed-M}
 M_{k,S,\ell}=
 \sum_{|I|=k}
 \frac{w_I^\ell
       \prod_{a\in S}\prod_{j\notin I}(1-qh_a/w_j)}
      {\prod_{i\in I,j\notin I}(1-w_j/w_i)}
 D_IU_S^{-1}.
\end{equation}
For a Laurent polynomial $f$ symmetric separately in
$(w_i)_{i\in I}$ and $(w_j)_{j\notin I}$, let $M_{k,S,f}$
denote the same sum with $w_I^\ell$ replaced by $f$; its coefficients
may involve the frozen characters $h_a^{\pm1}$.
Let $\gamma_a$ be the central coweight acting by $t$ on the $a$th summand of $N_{n,r}$
and trivially on the others. Then $\omega_n=\sum_a\gamma_a$, and
$\lambda_{k,S}=\omega_k-\sum_{a\in S}\gamma_a$.
The displayed operators are the dressed minuscule monopoles on
$\lambda_{k,S}$ for $0<k<n$. Write
$M_{k,S,\ell}^{\mathbf j}$ for the operator with dressing
$e_k^\ell\prod_b e_{j_b}$ in the variables $W_I$. Here $e_j$ denotes
the $j$th elementary symmetric polynomial in these $k$ variables,
with $e_0=1$ and $e_j=0$ for $j<0$ or $j>k$. In particular
$M_{k,\varnothing,\ell}=E_{k,\ell}$ and
\[
 P_a:=M_{0,\{a\},0}
       =\prod_{i=1}^n(1-qh_a/w_i)U_a^{-1}
\]
is the monopole operator indexed by $-\gamma_a$. For $k=n$,
$M_{n,S,f}=f(w)U_{S^c}$; in particular
$M_{n,S,\ell}=d_{\mathrm{raw}}^\ell U_{S^c}$ and $M_{n,[1,r],f}=f(w)$.

\begin{proposition}
\label{prop:gln-mixed-determinant-membership}
For every subset $S\subseteq[1,r]$, every $1\leq k<n$, and every
$\ell\in\mathbb Z$, the ordinary quantum cluster algebra
$\mathcal A_{n,r}$ contains $M_{k,S,\ell}$. It also contains
the operators $U_a$ and $P_a$ for every $a$.
For $\ell=0,1$, each $M_{k,S,\ell}$ is a quantum cluster variable
multiplied by an invertible frozen monomial and a unit
$\pm q^m$.
\end{proposition}

\begin{proof}
Put $\bar p_1=0$ for this argument. Reversing
the arrows indexed by $S$ replaces the mutable directions by
\[
 e_{u_a}^S=
 \begin{cases}
  \bar p_a-p_1,&a\in S,\\
  p_n-\bar p_a,&a\notin S,
 \end{cases}
 \qquad e_{y_k}^S=e_{y_k},\qquad e_{z_k}^S=e_{z_k}.
\]
The formula for $a=1$ is obtained by retaining the rank-one framing
character until the final specialization, as in the proof of
Theorem~\ref{thm:gln-upper-Coulomb}. In that extended lattice the
framing vector is an additional radical frozen direction. The
framing character splits off, so its specialization to one commutes
with every mutation; the following bases and pairings are written
in the resulting quotient lattice.

Write $a_S=\sum_{a\in S}\alpha_{u_a}$. An integral basis of
$\mathbb L_{n,r}$ satisfying the compatibility identities below is
\begin{equation}
\label{eq:gln-reversed-A-basis}
 \alpha_{u_a}^S=
 \begin{cases}-\alpha_{u_a},&a\in S,\\
               \alpha_{u_a},&a\notin S,
 \end{cases}
 \qquad
 \alpha_{y_k}^S=\alpha_{y_k}-a_S,\qquad
 \alpha_{z_k}^S=\alpha_{z_k}-a_S,
\end{equation}
with $\alpha_d^S=\alpha_d$ and
$\alpha_{h_a}^S=\alpha_{h_a}$. The change of basis is integral
and unimodular. Direct pairing gives
\[
 (e_i^S,\alpha_j^S)=\delta_{ij}
 \quad(i,j\text{ mutable}),\qquad
 (e_i^S,\alpha_d)=(e_i^S,\alpha_{h_a})=0.
\]
Indeed, for $a\in S$, the vector $e_{u_a}^S$ pairs to $-1$ with
$\alpha_{u_a}$ and to zero with $\alpha_{u_b}$ for $b\ne a$;
it pairs to $-1$ with each $\alpha_{y_k}$ and $\alpha_{z_k}$.
Subtracting $a_S$ therefore gives the asserted duality. The
directions $e_{y_k}$ and $e_{z_k}$ pair trivially with every $\alpha_{u_a}$.

By \cite[Lemma~5.6 and Corollary~5.7]{SS19}, reversing the arrow indexed
by $a$ is realized by the following Baxter sequence of seed mutations:
\[
 (u_a,z_{n-1},y_{n-1},\ldots,z_1,y_1).
\]
After this sequence the label $u_a$ occupies the old $y_1$
position, $y_k$ occupies the old $z_k$ position, and $z_k$
occupies the old $y_{k+1}$ position for $k<n-1$, with $z_{n-1}$
in the old $u_a$ position. The labels $u_b$ with $b\ne a$ are unchanged.
The same statement applies after reversing any other subset of
arrows, so successive mutation sequences reach the direction set $\mathbf e^S$.

Let $\widehat\alpha_j^S$ be the basis vectors obtained by the
corresponding signed lattice mutations and this relabeling. They have
the same pairings with $\mathbf e^S$ as the vectors in
\eqref{eq:gln-reversed-A-basis}. Consequently
\[
 \widehat\alpha_j^S-\alpha_j^S
 \in \bigoplus_{f\in I_{\mathrm{fr}}}\mathbb Z\alpha_f.
\]
Expansion in the full basis
\eqref{eq:gln-reversed-A-basis} shows that every mutable
coefficient of the difference vanishes. Frozen vectors are
unchanged by mutation. Thus, after the chart map induced by this mutation sequence, the
based monomial for each $\alpha_j^S$ is a cluster variable
multiplied by an invertible frozen monomial and a scalar $\pm q^m$, $m\in\mathbb Z$.
Applying $\rho^{-1}$ gives the same assertion in $\mathcal T_+$;
the scalar changes only by a sign.

It remains to determine these monomials under the Whittaker
transform $W$. Let $Q_S$ be the product of the bi-fundamental Baxter operators
of \cite[Def.~5.8]{SS19} for the reversed arrows, and put
\[
 \Phi_S=\prod_{a\in S}\prod_{i=1}^n\Psi_q(h_a/w_i),\qquad
 \Psi_q(q^2x)=(1-qx)\Psi_q(x).
\]
We use pullback from the reversed chart to the original chart.
In the convention of \cite[Corollary~5.9]{SS19}, conjugation by $Q_S$
maps the universally Laurent algebra of the original chart to that of
the reversed chart. Thus the pullback is $\operatorname{Ad}_{Q_S^{-1}}$.
Proposition~5.11 of \cite{SS19} gives
\begin{equation}
\label{eq:gln-inverse-Baxter-intertwining}
 W\circ\operatorname{Ad}_{Q_S^{-1}}
   =\operatorname{Ad}_{\Phi_S^{-1}}\circ W.
\end{equation}
For several arrows the factors commute by
\cite[Corollary~5.12]{SS19}. The Whittaker isomorphism for $N=0$
extends to the ambient fraction fields, where
\cite[equation~(5.21)]{SS19} gives the intertwining identity.
The factors $U_S^{-1}$ below are taken in these fields.

By the initial Whittaker formulas, the based monomials with
exponents $\alpha_{z_k}^S$ and $\alpha_{y_k}^S$ have images
$E_{k,0}U_S^{-1}$ and $E_{k,1}U_S^{-1}$, respectively, up to
units $\pm q^m$. The shift $D_IU_S^{-1}$ sends $h_a/w_j$ to
$q^2h_a/w_j$ precisely when $a\in S$ and $j\notin I$; for
$a\in S$ and $j\in I$ it fixes this ratio. The difference
equation therefore yields
\[
 \operatorname{Ad}_{\Phi_S^{-1}}
   (E_{k,\ell}U_S^{-1})=M_{k,S,\ell}.
\]
This proves the membership and cluster-variable assertions for
$\ell=0,1$. For $a\in S$, the same computation sends $U_a^{-1}$ to $P_a$;
it fixes $U_b$ for $b\notin S$.
The operators $U_a$ already occur in the initial seed.

Finally use the automorphism induced by $\mathsf T$ in
Proposition~\ref{prop:gln-translation-stability}. In the difference-operator realization it is
$\gamma(w_i)=w_i$, $\gamma(h_a)=h_a$,
$\gamma(D_i)=-qw_iD_i$, and $\gamma(L_a)=L_a$.
Normal ordering gives
\[
 \gamma(M_{k,S,\ell})=
 \begin{cases}
  (-q)^k M_{k,S,\ell+1},&1\notin S,\\
  (-q)^{n-k}d_{\mathrm{raw}}^{-1}M_{k,S,\ell+1},&1\in S.
 \end{cases}
\]
The factors $(-q)^k$ and $(-q)^{n-k}d_{\mathrm{raw}}^{-1}$ are invertible. Both $\gamma$ and $\gamma^{-1}$ preserve
$\mathcal A_{n,r}$, so all integer determinant dressings follow.
\end{proof}

\subsection{The remainder ideal}

Let $B=\mathbb Z[q^{\pm1},h_2^{\pm1},\ldots,h_r^{\pm1}]$, put
$h_1=1$, and fix $S\subseteq[1,r]$ of cardinality $s$.
For $W=(w_1,\ldots,w_k)$, write
\[
 R_k=B[W^{\pm1}]^{\mathfrak S_k},\qquad
 e_j=e_j(W),\qquad
 \chi_k(x)=\prod_{i=1}^k(x-w_i),\qquad
 F_S(x)=\prod_{a\in S}(x-qh_a).
\]
Let $J_{k,S}\subseteq R_k$ be the ideal generated by the
coefficients of the remainder of $F_S(x)$ on division by the monic
polynomial $\chi_k(x)$.

\begin{lemma}
\label{lem:remainder-ideal}
Let $2\leq k\leq n$. Suppose a subalgebra $A$ of the ambient
difference-operator algebra contains $B$, every $E_{1,m}$, and
$M_{k-1,S,f}$ for every $f\in R_{k-1}$.
Then $A$ contains $M_{k,S,g}$ for every $g\in J_{k,S}$.
\end{lemma}

\begin{proof}
For every integer $t$, define
\[
 H_t(W)=\sum_{b=1}^k
 \frac{w_b^t\prod_{a\in S}(1-qh_a/w_b)}
 {\prod_{i\ne b}(1-w_b/w_i)}.
\]
These are symmetric Laurent polynomials.  More precisely, let
$\lambda:R_k[x]/(\chi_k)\longrightarrow R_k$ take the coefficient
of $x^{k-1}$ in the unique representative of degree less than $k$.
Since $e_k$ is a unit, $x$ is invertible in this quotient.  Lagrange
interpolation gives
\begin{equation}
\label{eq:H-trace}
 H_t=(-1)^{k-1}e_k\,
       \lambda\bigl(x^{t-s-1}F_S(x)\bigr).
\end{equation}
This identity can be checked after adjoining the distinct roots and
therefore holds in $R_k$.

Let $f=e_k^\ell\prod_{b=1}^m e_{j_b}$, set
$\delta=\mathbf1_{1\in S}$, and put
$m_{p,\epsilon}=\ell-p+|\epsilon|+t$.
We have the identity
\begin{equation}
\label{eq:general-mixed-recurrence}
 M_{k,S,fH_t}
 =\sum_{p=0}^{k-1}\sum_{\epsilon\in\{0,1\}^m}
 (-q^2)^p q^{2\delta m_{p,\epsilon}}
 M_{k-1,S,\ell}^{(\mathbf j-\epsilon,p)}
 E_{1,m_{p,\epsilon}}.
\end{equation}
For $S=\varnothing$ and $t=0$, this is the identity in
\cite[Lemma~2.18]{CW19}, in our normalization.
To check this identity, fix the $(k-1)$-element set $J$ in the
first factor and the index $b$ in the second. The shift $D_JU_S^{-1}$
sends $w_b$ to $q^{2\mathbf1_{b\in J}-2\delta}w_b$.
The factor $q^{2\delta m_{p,\epsilon}}$ compensates for the common scaling
$w_b\mapsto q^{-2\delta}w_b$ in the determinant dressing.
With $z=q^2w_b$ for $b\in J$ and $z=w_b$ otherwise,
the sum over $p$ is
$\prod_{i\in J}(1-q^2w_i/z)$. It vanishes when $b\in J$ and
cancels the shifted denominator from the second factor when $b\notin J$.
For $I=J\sqcup\{b\}$, the expansion
$e_j(W_I)=e_j(W_J)+w_be_{j-1}(W_J)$ gives $f(W_I)$.
Summing the remaining coefficients over $b\in I$ gives
$H_t(W_I)$. The calculation also applies at $k=n$.
Since these monomials $f$ span $R_k$ over $B$,
the hypotheses give $M_{k,S,fH_t}\in A$ for all $f\in R_k$.

The pairing $(u,v)\mapsto\lambda(uv)$ on $R_k[x]/(\chi_k)$ is
unimodular.  In the basis $1,x,\ldots,x^{k-1}$ its matrix has
zero entries for row plus column less than $k-1$, and has $1$ on
the antidiagonal.  Consequently the $k$ elements
$\lambda(x^jF_S)$, $0\leq j<k$, generate the same ideal as the
coefficients of the remainder of $F_S$.  By
\eqref{eq:H-trace}, these differ by a common unit from
$H_{s+1},\ldots,H_{s+k}$. Equation~\eqref{eq:general-mixed-recurrence}
therefore gives $M_{k,S,g}\in A$ for every $g\in J_{k,S}$.
\end{proof}

The quotient $R_k/J_{k,S}$ is the universal factorization ring in
which $\chi_k$ divides $F_S$.  Three cases suffice below:
\begin{enumerate}
\item If $s<k$, the remainder is the monic polynomial $F_S$ itself,
so $J_{k,S}=R_k$.
\item If $s=k$, division has zero remainder precisely when
$\chi_k=F_S$.  Thus $R_k/J_{k,S}$ is generated by $B$.
\item If $s=k+1$, write $F_S=\chi_k(x)(x-z)$ in the quotient.
Comparing constant coefficients gives
$z=(\prod_{a\in S}qh_a)e_k^{-1}$.  Comparing the other
coefficients successively expresses each $e_j$ as a polynomial
over $B$ in $z$.  Thus $R_k/J_{k,S}$ is generated as a
$B$-module by the integer powers of $e_k$.
\end{enumerate}
\subsection{Dressings in two variables}

\begin{lemma}
\label{lem:point-multiplication}
Let $a\notin T$ and set $S=T\cup\{a\}$.  For every $f\in R_k$, one has
\begin{equation}
\label{eq:point-product}
 M_{k,T,f}P_a
 =M_{k,S,\,f\prod_{i=1}^k(1-q^{-1}h_a/w_i)}.
\end{equation}
\end{lemma}

\begin{proof}
Put $K_a=\prod_i(1-qh_a/w_i)$, so $P_a=K_aU_a^{-1}$.
Since $a\notin T$, the shift $U_T^{-1}$ fixes each ratio
$h_a/w_i$.  Indeed, if $1\in T$, its action on $w_i$ and on
$h_a$ has the same factor $q^{-2}$; if $a=1$, then $1\notin T$
and both are fixed.  The shift $D_I$ replaces $h_a/w_i$ by
$q^{-2}h_a/w_i$ exactly for $i\in I$.  Therefore
\[
 (D_IU_T^{-1})K_a
 =\Bigl(\prod_{i\in I}(1-q^{-1}h_a/w_i)
        \prod_{j\notin I}(1-qh_a/w_j)\Bigr)D_IU_T^{-1}.
\]
All $U_a$ commute.  Substituting this identity into the product
gives the assertion, with no shift of the dressing $f$.
\end{proof}

\begin{proposition}
\label{prop:rank-two-dressings}
For every $n\geq2$, every $r\geq1$, and every $S\subseteq[1,r]$,
the ordinary cluster algebra contains $M_{2,S,f}$ for every $f\in R_2$.
\end{proposition}

\begin{proof}
Induct on $|S|$.  For $S=\varnothing$, Lemma~\ref{lem:remainder-ideal}
gives the result: $F_\varnothing=1$ makes $J_{2,\varnothing}=R_2$,
and all dressings for $k=1$ are determinant dressings. Choose $a\in S$, put $T=S\setminus\{a\}$,
and set $c=q^{-1}h_a$, a unit of $B$.  The induction hypothesis
and \eqref{eq:point-product} give every dressing in the
principal ideal generated by
\[
 G_a=\prod_{i=1}^2(1-c/w_i)
     =1-ce_1e_2^{-1}+c^2e_2^{-1}.
\]
In $R_2/(G_a)$ one has $e_1=c+c^{-1}e_2$.  Hence this quotient
is generated as a $B$-module by the integer powers of $e_2$.
All corresponding determinant-dressed operators are already in
the cluster algebra.  For $k=n=2$, this last assertion follows
directly from $M_{n,S,\ell}=d_{\mathrm{raw}}^\ell U_{S^c}$.
Every dressing is therefore the sum of a $B$-linear combination
of integer powers of $e_2$ and an element of $(G_a)$; the corresponding
operators all belong to the cluster algebra.
\end{proof}

\subsection{Removing one linear factor}

Write $\delta_T=\mathbf 1_{1\in T}$.  For a subset $R\subseteq[1,r]$, choose
$a\in R$ and put $S=R\setminus\{a\}$. Let $2\leq k\leq n$ and $m=k-1$.
The product identity below
will give operators $M_{k,R,g}$ with dressings $g\in J_{k,S}$,
where $F_S(x)=F_R(x)/(x-qh_a)$.
For $m$ variables $W=(w_1,\ldots,w_m)$, set
$e_m=e_m(W)$, $\chi_W(z)=\prod_{i=1}^m(z-w_i)$ and $c=q^{-1}h_a$.
Define
\[
 Q_{W,a}(z)=\frac{(-1)^m}{e_m}
 \frac{z\chi_W(c)-q^2c\chi_W(z/q^2)}{z-q^2c}
 =\sum_{j=0}^{m-1}Q_{a,j}(W)z^j.
\]
The numerator vanishes at $z=q^2c$, so this is a polynomial with
coefficients in $R_m$.

\begin{proposition}
\label{prop:single-wall-erasure}
For every $f\in R_m$ and $t\in\mathbb Z$,
put, with all displayed scalar factors on the left,
\[
 C_t(f)=q^{2\delta_S t}M_{m,S,f}M_{1,\{a\},t}
 -\sum_{j=0}^{m-1}q^{2\delta_R(t+j)}
 M_{m,R,fQ_{a,j}}E_{1,t+j}.
\]
For $W=(w_1,\ldots,w_k)$, write $W_{\widehat b}=(w_i)_{i\ne b}$.
Define $\Psi_\nu$ for $\nu\in\mathbb Z$ by the formula below. Then
\begin{equation}\label{eq:single-wall-erasure}
 C_t(f)=(-1)^m q^{1-2m}h_a M_{k,R,\,e_k^{-1}\Psi_{t+m}(f)},
 \qquad
 \Psi_\nu(f)(W)=\sum_{b=1}^k
 \frac{f(W_{\widehat b})w_b^\nu
 \prod_{c'\in S}(1-qh_{c'}/w_b)}
 {\prod_{i\ne b}(1-w_b/w_i)}.
\end{equation}
\end{proposition}

\begin{proof}
Expand the two products in normal order.  Fix the first selected set $I$,
of size $m$, and the second selected index $b$.  The factors
$q^{2\delta_S t}$ and $q^{2\delta_R(t+j)}$ cancel precisely the common
scaling of the determinant dressing in the second factor; this also covers $a=1$.
If $b\in I$, the remaining dressing is evaluated at $q^2w_b$ and
\[
 Q_{W_I,a}(q^2w_b)=\prod_{i\in I\setminus\{b\}}(1-c/w_i).
\]
Thus every term with $b\in I$ cancels. For $b\notin I$,
the remaining coefficients satisfy
\[
 \prod_{i\in I}(1-c/w_i)-(1-qh_a/w_b)Q_{W_I,a}(w_b)
 =(-1)^m q^{1-2m}h_a
 \frac{\prod_{i\in I}(w_b-q^2w_i)}{w_b e_m(W_I)}.
\]
This follows immediately on multiplying the defining identity for $Q$ by
$w_b-q^2c$. The factor $\prod_{i\in I}(w_b-q^2w_i)$ cancels the shifted denominator
$\prod_{i\in I}(1-q^2w_i/w_b)$ from the second factor.
The remaining denominator is
$\delta_{I\cup\{b\}}\prod_{i\in I}(1-w_b/w_i)$, and
$w_b^{m-1}/e_m(W_I)=w_b^m/e_k(W_{I\cup\{b\}})$.
The remaining product over $S$ is
$\prod_{c'\in S}(1-qh_{c'}/w_b)$.
Collecting the terms with the same $k$-element set gives the formula.
At $k=n$ the products indexed by variables outside $I\cup\{b\}$ are empty.
\end{proof}

\begin{lemma}\label{lem:erased-remainder-ideal}
Assume $M_{k-1,R,f}$ and $M_{k-1,S,f}$ belong to the ordinary
cluster algebra for every $f\in R_{k-1}$. Then
$M_{k,R,g}$ belongs to it for every $g\in J_{k,S}$, where
$J_{k,S}$ is the ideal of the coefficients of the remainder of
$F_S(x)=\prod_{a\in S}(x-qh_a)$ modulo $\chi_k(x)$.
\end{lemma}
\begin{proof}
Proposition~\ref{prop:single-wall-erasure} produces operators with dressings $e_k^{-1}\Psi_\nu(f)$
for arbitrary $f$ and $\nu\in\mathbb Z$.
For $g=e_k^\ell\prod_\alpha e_{j_\alpha}$, expand separately for each
$b\in\{1,\ldots,k\}$:
\[
 g(W)=e_{k-1}(W_{\widehat b})^\ell w_b^\ell
 \prod_\alpha\bigl(e_{j_\alpha}(W_{\widehat b})
       +w_b e_{j_\alpha-1}(W_{\widehat b})\bigr).
\]
Apply this expansion to $e_k g$ to cancel the prefactor $e_k^{-1}$
in \eqref{eq:single-wall-erasure}.
Applying \eqref{eq:single-wall-erasure} term by term and shifting $\nu$
produces operators with dressings $gH_t$, for the Laurent polynomials
$H_t$ associated with $F_S$. The unimodular pairing in
Lemma~\ref{lem:remainder-ideal} shows that the $H_t$ generate $J_{k,S}$ as an ideal.
Each ideal multiple is therefore obtained from the displayed
convolution identities.
\end{proof}

\begin{theorem}
\label{thm:four-wall-dressings}
For every $n\geq2$ and $r\geq1$, if $|S|\leq5$, then
$M_{k,S,f}$ belongs to the ordinary cluster algebra for every
$1\leq k\leq n$ and every $f\in R_k$.
\end{theorem}

\begin{proof}
The assertion has been proved for $k=1,2$ and every subset $S$.
Induct on $k\geq3$.  If $S=\varnothing$, use Lemma~\ref{lem:remainder-ideal}
with $F_\varnothing=1$.
Otherwise choose $a\in S$ and put $T=S\setminus\{a\}$.
The induction hypothesis and Lemma~\ref{lem:erased-remainder-ideal}
give $M_{k,S,g}\in\mathcal A_{n,r}$ for every $g\in J_{k,T}$.
Since $|T|\leq4$ and $k\geq3$, the quotient $R_k/J_{k,T}$ is zero,
is generated by $B$, or is spanned over $B$ by integer powers of $e_k$. These are
exactly the three factorization cases following
Lemma~\ref{lem:remainder-ideal}. The operators $M_{k,S,\ell}$,
$\ell\in\mathbb Z$, belong to $\mathcal A_{n,r}$ for $k<n$ by
Proposition~\ref{prop:gln-mixed-determinant-membership}
and for $k=n$ by $M_{n,S,\ell}=d_{\mathrm{raw}}^\ell U_{S^c}$.
Thus every dressing is the sum of a $B$-linear combination of
integer powers of $e_k$ and an element of $J_{k,T}$; the corresponding
operators all belong to the cluster algebra.
\end{proof}

\subsection{Completion of the generation proof}

\begin{lemma}
\label{lem:gen-chamber}
In every closed chamber of the hyperplane arrangement defined by the weights of $N_{n,r}$ inside the
dominant Weyl chamber, the semigroup of integral coweights is generated
by the members of
\begin{equation}
\label{eq:gen-ray-set}
 \{\gamma_a,-\gamma_a:1\leq a\leq r\}
 \ \cup\ \{\lambda_{k,S}:1\leq k<n,\ S\subseteq[1,r]\}
\end{equation}
that lie in that chamber.
\end{lemma}

\begin{proof}
Represent a dominant coweight uniquely by
$(\mu_1,\ldots,\mu_n;m_1,\ldots,m_r)$ with
$\mu_1\geq\cdots\geq\mu_n=0$. The pairings with the weights of $N_{n,r}$ are
$\mu_i+m_a$.  Put $D=\mu_1$ and, for $1\leq t\leq D$, define
\[
 k(t)=\#\{i:\mu_i\geq t\},\qquad
 S_t=\{a:-m_a\geq t\}.
\]
Then $1\leq k(t)<n$ and
\begin{equation}
\label{eq:gen-layer-decomposition}
 (\mu;m)=\sum_{t=1}^D\lambda_{k(t),S_t}
 +\sum_{m_a>0}m_a\gamma_a
 +\sum_{m_a<-D}(-D-m_a)(-\gamma_a).
\end{equation}
To verify the formula, compare the coordinates.  The $\mu_i$
coordinate is the number of $t\leq\mu_i$. In the $m_a$ coordinate, the first sum
has coordinate $-\min(D,\max(0,-m_a))$, and the remaining two sums give
exactly the required correction.

The pairing of the weight indexed by $(i,a)$ with $\lambda_{k(t),S_t}$ is
$\mathbf1_{t\leq\mu_i}-\mathbf1_{t\leq-m_a}$.
As $t$ varies, these differences are either zero or have the same sign as
$\mu_i+m_a$. Each nonzero term in the last two sums has pairings of
that same sign with the weights of the summand of $N_{n,r}$ on which
it acts nontrivially.
Moreover each $\lambda_{k(t),S_t}$ is dominant. Thus every
summand lies in every closed chamber containing the given coweight,
including all chamber faces.
\end{proof}

\begin{proof}[Proof of Theorem~\ref{thm:integral-generation}]
First suppose $n=2$, with no restriction on $r$.
Proposition~\ref{prop:rank-two-dressings}, applied at $k=n=2$ and $S=[1,r]$, gives
\[
 M_{2,[1,r],f}=f(w_1,w_2)\in\mathcal A_{2,r}
\]
for every $f\in R_2$. Since $R_1=B[w_1^{\pm1}]$,
Proposition~\ref{prop:gln-mixed-determinant-membership}
also gives $M_{1,S,f}\in\mathcal A_{2,r}$ for every
$S\subseteq[1,r]$ and $f\in R_1$.

For arbitrary $n\geq2$ and $r\leq5$, apply
Theorem~\ref{thm:four-wall-dressings} at $k=n$ and $S=[1,r]$:
\[
 M_{n,[1,r],f}=f(w_1,\ldots,w_n)\in\mathcal A_{n,r}.
\]
The same theorem supplies $M_{k,S,f}$ for every $f\in R_k$ and every
$\lambda_{k,S}$. Thus, in either range of the theorem, $\mathcal A_{n,r}$
contains $R_n$ and every $M_{k,S,f}$ with $1\leq k<n$,
$S\subseteq[1,r]$, and $f\in R_k$.

For each $\lambda_{k,S}$ with $0<k<n$, the $GL_n$ factor of the stabilizer Levi is
$GL_k\times GL_{n-k}$, and its dressing ring is
$B[w_1^{\pm1},\ldots,w_n^{\pm1}]^{\mathfrak S_k\times\mathfrak S_{n-k}}$. By the calculation for symmetric Laurent polynomials
in \cite[proof of Proposition~2.7(2)]{SS19}, this ring is generated
over $R_n$ by the elementary symmetric functions in $w_1,\ldots,w_k$
and $e_k^{-1}$.
It follows that $M_{k,S,f}\in\mathcal A_{n,r}$ for every dressing $f$
in the representation ring of this stabilizer Levi.
Multiplication by a character is on the left, so it introduces no
additional difference-operator shift.
The monopole operators indexed by $\pm\gamma_a$ belong to the cluster algebra by
Proposition~\ref{prop:gln-mixed-determinant-membership}; their dressing
modules have rank one over $R_n$.

The generation result of
\cite[Proposition~2.3 and the discussion following it]{SS19}, applied
to Lemma~\ref{lem:gen-chamber}, now puts every Coulomb-branch generator
in $\mathcal A_{n,r}$. The trivially acting framing character splits
off in the Coulomb branch and in its coefficient ring, as in
Theorem~\ref{thm:gln-upper-Coulomb}; quotienting by that independent
character minus one preserves the generating equality.
Consequently $\mathcal K_{n,r}\subseteq\mathcal A_{n,r}$.
The opposite containment is the integral categorical cluster embedding
constructed above. The upper-algebra realization in
Theorem~\ref{thm:gln-upper-Coulomb} gives the stated equality with
$\mathcal U_{n,r}$.
\end{proof}

\section{Monoidal categorification}
\label{sec:complete-categorification}
\label{sec:family-categorification}

Write $\mathcal C_{n,r}$ for the Koszul-perverse heart of
$(G_{n,r},N_{n,r})$, with integral loop shifts and the normalization
of Section~\ref{sec:loop-rotation-cover}.

\begin{theorem}
\label{thm:family-categorification}
Suppose either $n=2$ and $r\geq1$, or $n\geq2$ and $1\leq r\leq5$.
For the compatible pair of Theorem~\ref{thm:gln-compatible-pair},
the map carrying the initial variables to the classes of
$\mathsf U_a$, $\mathsf Y_k$, $\mathsf Z_k$, $\mathsf D$, and
$\mathsf H_a$ induces an isomorphism of $\Z[\vloop^{\pm1}]$-algebras
\[
 \mathcal A_{n,r}=\mathcal U_{n,r}
 \xrightarrow{\ \sim\ } K_0(\mathcal C_{n,r}).
\]
Here the $r$ frozen character variables are inverted.
Every quantum cluster variable is represented by a real simple object,
every localized quantum cluster monomial by a simple normalized
convolution, and every finite mutation sequence by short exact
sequences in $\mathcal C_{n,r}$.
Consequently $\mathcal C_{n,r}$ is a monoidal categorification of
$\mathcal A_{n,r}$ in the sense of
Definition~\ref{def:monoidal-categorification}.
After adjoining half-integer loop shifts, it is also a monoidal
categorification in the convention of \cite[Def.~4.26]{CW19}.
\end{theorem}

\begin{proof}
Theorem~\ref{thm:integral-generation} gives the ring equality in both
ranges. Proposition~\ref{prop:gln-simple-frame} identifies the
initial classes, and Theorem~\ref{thm:cat-all-mutations} supplies the
categorical realization of every finite mutation sequence.
The commutation form is even and remains even under mutation, so all
normalized cluster monomials and exchange sequences use integral
loop shifts. The frozen objects are tensor-invertible in the heart,
which permits arbitrary integral exponents on these objects.

For the last assertion, put $R=\Z[\vloop^{\pm1}]$ and
$S=\Z[\vloop^{\pm1/2}]$.
Lemma~\ref{lem:CW-half-shift-extension} identifies
\[
 S\otimes_R\mathcal A_{n,r}
 \xrightarrow{\sim}S\otimes_RK_0(\mathcal C_{n,r})
 \xrightarrow{\sim}K_0(\mathcal C_{n,r}^{1/2}).
\]
The corresponding scalar extension of the quantum seed is the
Cautis--Williams convention of \eqref{eq:CW-half-scalar-extension}.
Every based cluster monomial already lies in $\mathcal A_{n,r}$
and is represented by a simple object of $\mathcal C_{n,r}$;
this object remains simple after a formal half-integer shift.
\end{proof}

For all $n\geq2$ and $r\geq1$, Theorems~\ref{thm:gln-upper-Coulomb}
and~\ref{thm:cat-all-mutations} identify $\mathcal U_{n,r}$ with
$K_0(\mathcal C_{n,r})$ and realize every cluster mutation.
Generation by the ordinary cluster algebra remains open for
$n\geq3$ and $r\geq6$.
The restriction $r\leq5$ enters only in the induction for $k\geq3$
in Theorem~\ref{thm:four-wall-dressings}: for $a\in S$, the quotient
$R_k/J_{k,S\setminus\{a\}}$ obtained by removing $x-qh_a$ from
$F_S(x)$ is spanned over $B$ by integer powers of $e_k$.
The first case beyond this argument is $k=3$, $|S|=6$:
$F_{S\setminus\{a\}}$ has degree $5=k+2$, outside the three
factorization cases used there.

\section*{Declaration of generative AI use}
Generative artificial intelligence (GenAI) tools were used to discuss
mathematical ideas, assist in checking the manuscript, and proofread the
text. The author takes full responsibility for the content of this article.

\begingroup
\renewcommand{\bibliofont}{\footnotesize}
\raggedright
\raggedbottom

\bibliographystyle{alpha}
\bibliography{references}

\begin{thebibliography}{KKKO18}

\bibitem[AB10]{AB10}
Dmitry Arinkin and Roman Bezrukavnikov.
\newblock Perverse coherent sheaves.
\newblock {\em Moscow Mathematical Journal}, 10(1):3--29, 2010.

\bibitem[AKS25]{AKS25}
Dylan G.~L. Allegretti, Hyun~Kyu Kim, and Peng Shan.
\newblock Monoidal categorification of genus zero skein algebras.
\newblock {\em arXiv preprint arXiv:2505.13332}, 2025.
\newblock arXiv:2505.13332v1.

\bibitem[Bar15]{Barwick15}
Clark Barwick.
\newblock On exact {$\infty$}-categories and the {Theorem of the Heart}.
\newblock {\em Compositio Mathematica}, 151(11):2160--2186, 2015.

\bibitem[BFM05]{BFM05}
Roman Bezrukavnikov, Michael Finkelberg, and Ivan Mirkovi\'c.
\newblock Equivariant homology and {$K$}-theory of affine {Grassmannians} and
  {Toda} lattices.
\newblock {\em Compositio Mathematica}, 141(3):746--768, 2005.

\bibitem[BFN18]{BFN18}
Alexander Braverman, Michael Finkelberg, and Hiraku Nakajima.
\newblock Towards a mathematical definition of {Coulomb} branches of
  3-dimensional {$\mathcal N=4$} gauge theories, {II}.
\newblock {\em Advances in Theoretical and Mathematical Physics},
  22(5):1071--1147, 2018.

\bibitem[BFN19]{BFN19}
Alexander Braverman, Michael Finkelberg, and Hiraku Nakajima.
\newblock Ring objects in the equivariant derived {Satake} category arising
  from {Coulomb} branches.
\newblock {\em Advances in Theoretical and Mathematical Physics},
  23(2):253--344, 2019.
\newblock With an appendix by Gus Lonergan.

\bibitem[BFZ05]{BFZ05}
Arkady Berenstein, Sergey Fomin, and Andrei Zelevinsky.
\newblock Cluster algebras {III}: Upper bounds and double {Bruhat} cells.
\newblock {\em Duke Mathematical Journal}, 126(1):1--52, 2005.

\bibitem[BGS96]{BGS96}
Alexander Beilinson, Victor Ginzburg, and Wolfgang Soergel.
\newblock Koszul duality patterns in representation theory.
\newblock {\em Journal of the American Mathematical Society}, 9(2):473--527,
  1996.

\bibitem[BZ05]{BZ05}
Arkady Berenstein and Andrei Zelevinsky.
\newblock Quantum cluster algebras.
\newblock {\em Advances in Mathematics}, 195(2):405--455, 2005.

\bibitem[CHL24]{CHL24}
Wen Chang, Min Huang, and Jian-Rong Li.
\newblock Quasi-homomorphisms of quantum cluster algebras.
\newblock {\em Journal of Algebra}, 638:506--534, 2024.
\newblock Numbered citations refer to arXiv:2303.13030v3.

\bibitem[CW19]{CW19}
Sabin Cautis and Harold Williams.
\newblock Cluster theory of the coherent {Satake} category.
\newblock {\em Journal of the American Mathematical Society}, 32(3):709--778,
  2019.

\bibitem[CW23]{CWInd23}
Sabin Cautis and Harold Williams.
\newblock Ind-geometric stacks.
\newblock {\em arXiv preprint arXiv:2306.03043}, 2023.
\newblock Version 2, revised 10 January 2024.

\bibitem[CW25]{CWTame23}
Sabin Cautis and Harold Williams.
\newblock Tamely presented morphisms and coherent pullback.
\newblock {\em Mathematische Zeitschrift}, 311(4), 2025.
\newblock Article 88.

\bibitem[CW26]{CW23}
Sabin Cautis and Harold Williams.
\newblock Canonical bases for {Coulomb} branches of 4d {$\mathcal N=2$} gauge
  theories.
\newblock {\em Inventiones Mathematicae}, 245(2):797--869, 2026.

\bibitem[FG09]{FG09}
Vladimir~V. Fock and Alexander~B. Goncharov.
\newblock Cluster ensembles, quantization and the dilogarithm.
\newblock {\em Annales Scientifiques de l'\'Ecole Normale Sup\'erieure},
  42(6):865--930, 2009.

\bibitem[FST08]{FST08}
Sergey Fomin, Michael Shapiro, and Dylan Thurston.
\newblock Cluster algebras and triangulated surfaces. {Part I}: Cluster
  complexes.
\newblock {\em Acta Mathematica}, 201(1):83--146, 2008.

\bibitem[FST12]{FSTFinite}
Anna Felikson, Michael Shapiro, and Pavel Tumarkin.
\newblock Skew-symmetric cluster algebras of finite mutation type.
\newblock {\em Journal of the European Mathematical Society}, 14(4):1135--1180,
  2012.

\bibitem[FZ02]{FZ02}
Sergey Fomin and Andrei Zelevinsky.
\newblock Cluster algebras {I}: Foundations.
\newblock {\em Journal of the American Mathematical Society}, 15(2):497--529,
  2002.

\bibitem[HL10]{HL10}
David Hernandez and Bernard Leclerc.
\newblock Cluster algebras and quantum affine algebras.
\newblock {\em Duke Mathematical Journal}, 154(2):265--341, 2010.

\bibitem[Kas02]{Kas02}
Masaki Kashiwara.
\newblock On level-zero representations of quantized affine algebras.
\newblock {\em Duke Mathematical Journal}, 112(1):117--175, 2002.

\bibitem[KKKO15]{KKKO15}
Seok-Jin Kang, Masaki Kashiwara, Myungho Kim, and Se-jin Oh.
\newblock Simplicity of heads and socles of tensor products.
\newblock {\em Compositio Mathematica}, 151(2):377--396, 2015.

\bibitem[KKKO18]{KKKO18}
Seok-Jin Kang, Masaki Kashiwara, Myungho Kim, and Se-jin Oh.
\newblock Monoidal categorification of cluster algebras.
\newblock {\em Journal of the American Mathematical Society}, 31(2):349--426,
  2018.
\newblock Numbered citations use the published version; arXiv:1412.8106v1 has
  earlier numbering.

\bibitem[Li25]{Li25}
Enhan Li.
\newblock Skein and cluster algebras of punctured surfaces.
\newblock {\em arXiv preprint arXiv:2503.15037}, 2025.
\newblock arXiv:2503.15037v4, revised 25 December 2025.

\bibitem[Liu26]{LiuThesis}
Tianle Liu.
\newblock {\em A Categorified Cluster Structure on the {Coulomb} Branch of
  {$\mathcal{N}=2$} {$U(2)$} {SQCD} with {$N_f=2$}}.
\newblock PhD thesis, University of Southern California, 2026.

\bibitem[Nak16]{Nak16}
Hiraku Nakajima.
\newblock Towards a mathematical definition of {Coulomb} branches of
  3-dimensional {$\mathcal N=4$} gauge theories, {I}.
\newblock {\em Advances in Theoretical and Mathematical Physics},
  20(3):595--669, 2016.

\bibitem[SS19]{SS19}
Gus Schrader and Alexander Shapiro.
\newblock {$K$}-theoretic {Coulomb} branches of quiver gauge theories and
  cluster varieties.
\newblock {\em arXiv preprint arXiv:1910.03186}, 2019.
\newblock arXiv:1910.03186v2, revised 5 January 2026.

\bibitem[Zhu17]{Zhu17}
Xinwen Zhu.
\newblock An introduction to affine {Grassmannians} and the geometric {Satake}
  equivalence.
\newblock In {\em Geometry of Moduli Spaces and Representation Theory},
  volume~24 of {\em IAS/Park City Mathematics Series}, pages 59--154. American
  Mathematical Society, Providence, RI, 2017.

\end{thebibliography}

\endgroup

\end{document}